\documentclass{article}

\usepackage{amsmath,amssymb,amsthm,bm}
\usepackage{algorithm,algorithmic}
\usepackage{authblk}
\usepackage[a4paper]{geometry}
\usepackage{graphicx}
\usepackage[bookmarksnumbered,pdfencoding=auto,psdextra]{hyperref}
\usepackage{comment}
\usepackage{xcolor}
\usepackage{multirow}
\usepackage{setspace}

\graphicspath{{fig/}}

\newenvironment{keywords}{\medskip\textbf{Keywords:}}{}
\newenvironment{AMS}{\medskip\textbf{AMS subject classifications (2020).}}{}

\theoremstyle{plain}
\newtheorem{theorem}{Theorem}
\newtheorem{lemma}{Lemma}
\newtheorem{corollary}{Corollary}
\newtheorem{example}{Example}

\theoremstyle{remark}
\newtheorem{remark}{Remark}

\newcommand*{\set}[1]{\left\lbrace#1\right\rbrace}

\newcommand*{\trans}{^{\top}}
\newcommand*{\herm}{^*}

\newcommand*{\abs}[1]{\bigl\lvert#1\bigr\rvert}
\newcommand*{\absbig}[1]{\Bigl\lvert#1\Bigr\rvert}

\newcommand*{\mi}{\mathrm i}
\newcommand*{\mj}{\mathrm j}
\newcommand*{\mk}{\mathrm k}
\def\adots{\mathinner{\mkern2mu\raise1pt\hbox{.}\mkern2mu
    \raise4pt\hbox{.}\mkern2mu\raise7pt\hbox{.}\mkern1mu}}

\DeclareMathOperator{\fl}{f{}l}

\newcommand*{\macheps}{\bm u}
\newcommand*{\op}{\circ}
\newcommand*{\bigO}{O}
\newcommand*{\gammanw}{\gamma_n^{\mathrm w}}
\newcommand*{\gammanh}{\gamma_n^{\mathrm h}}
\newcommand*{\gamman}{\gamma_n}

\newcommand*{\epsadd}{\varepsilon_{\mathtt{add}}}

\newcommand*{\epsdiv}{\varepsilon_{\mathtt{div}}}
\newcommand*{\epsdot}{\varepsilon_{\mathtt{dot}}}
\newcommand*{\epsgemm}{\varepsilon_{\mathtt{gemm}}}
\newcommand*{\epsgemmadd}{\varepsilon_{\mathtt{gemm-add}}}

\newcommand*{\epstrfvec}{\varepsilon_{\mathtt{getf2}}}
\newcommand*{\epstrf}{\varepsilon_{\mathtt{getrf}}}
\newcommand*{\epsmult}{\varepsilon_{\mathtt{mult}}}

\newcommand*{\epstrsv}{\varepsilon_{\mathtt{trsv}}}

\newcommand{\Acal}{\mathcal{A}}
\newcommand{\Acalhat}{\widehat{\mathcal{A}}}

\newcommand{\LL}{\mathcal{L}}
\newcommand{\LLhat}{\widehat{\mathcal{L}}}

\newcommand{\UU}{\mathcal{U}}
\newcommand{\UUhat}{\widehat{\mathcal{U}}}

\title{A framework for trackable and modular rounding error analysis}
\author[1]{Erin Carson}
\author[1]{Yuxin Ma}
\author[2,3]{Meiyue Shao}
\affil[1]{Department of Numerical Mathematics, Faculty of Mathematics and Physics, Charles University, Sokolovsk\'{a} 49/83, 186 75 Praha 8, Czechia}
\affil[2]{School of Data Science, Fudan University, Shanghai 200433, China}
\affil[3]{Shanghai Key Laboratory for Contemporary Applied Mathematics, Fudan
University, Shanghai 200433, China}

\begin{document}

\maketitle

\begin{abstract}
Robust and meaningful rounding error analysis is becoming increasingly challenging as modern
numerical computations involve a rapidly evolving variety of arithmetic formats,
precisions, and computational kernels.
We propose a trackable and modular framework for rounding error analysis,
together with a new notation for bounding accumulated rounding errors.
The analysis is trackable in the sense that the origin and propagation of
individual error terms remain explicit in the final result, and modular in
the sense that local error bounds can be replaced according to the arithmetic,
precision, or computational kernel without repeating the structural analysis.
We also introduce the notation \(\gamman=(1-\macheps)^{-n}-1\),
which provides an alternative to the classical
\(\gammanh=n\macheps/(1-n\macheps)\) notation.
A key feature is that our notation does not require the restriction \(n\macheps<1\) and conveniently handles
errors arising from division.
Using block LU factorization as a  representative example,
we show that the framework distinguishes the error sources of different
algorithmic variants and permits improved kernels, such as compensated
summation and accurate matrix multiplication, to be incorporated directly
into the resulting error bounds.
\medskip

\begin{keywords}
Rounding error analysis, numerical linear algebra, matrix computations
\end{keywords}

\begin{AMS}
65G50
\end{AMS}
\end{abstract}

\section{Introduction}
\label{sec:introduction}
Rounding errors are an inevitable consequence of floating-point arithmetic
and play a fundamental role in determining the reliability of numerical
computations.
Since the pioneering work of Wilkinson~\cite{Wilkinson1963}, rounding error
analysis has provided a theoretical basis for understanding numerical
stability and for assessing the reliability of computed results.
The arithmetic environment in which numerical algorithms are executed,
however, is becoming increasingly heterogeneous.
Recent developments include new IEEE and non-IEEE arithmetic formats, such as
the parameterized family of formats being developed in the IEEE P3109 draft
standard~\cite{FitzgibbonEtAl2026} and the Open Compute Project (OCP)
Microscaling (MX) formats~\cite{OCPMX2023,RouhaniEtAl2023}, mixed-precision
algorithms~\cite{HighamMary2022}, software emulation of custom low-precision
arithmetic~\cite{HP2019}, stochastic rounding~\cite{CrociEtAl2022}, and
specialized hardware such as GPU tensor cores~\cite{BlanchardEtAl2020}.
These developments change not only the unit round-off, but also the numerical
formats, rounding modes, accumulation rules, and precisions assigned to
individual operations.
Consequently, an error analysis developed for one arithmetic model or
implementation often has to be substantially revisited when any of these
ingredients changes.
This raises a basic question: how can rounding error analysis be organized so
that it remains meaningful, interpretable, and reusable as arithmetic and
hardware continue to evolve?

We advocate a framework for \emph{trackable and modular rounding error analysis}.
By trackability, we mean that local errors are retained according to the
operations that generate them and are propagated explicitly to the final
result.
Thus, each term in the final error representation can be traced back to the
operation or computational kernel from which it originates.
By modularity, we mean that the structural propagation of these errors is
separated from the particular bounds imposed on the local errors.
Once the structural analysis has been established, the local error bounds can
be replaced according to the arithmetic, precision, kernel, or implementation
being used, without repeating the global derivation.
For example, the same structural analysis can incorporate bounds for standard
or compensated summation, standard or Ozaki-style accurate matrix
multiplication, and different mixed-precision kernels.
In short, trackability concerns the provenance of the error, whereas
modularity concerns the substitutability of the local error bounds.
This distinction is important because a compact final bound alone may conceal
which part of a computation is responsible for the observed loss of accuracy.
Two algorithms may have final bounds of comparable size while their errors
originate from entirely different computational components.
Identifying these sources is essential when one seeks to repair an unstable
algorithm, improve a particular kernel, or selectively increase the precision
of only the critical operations.

Wilkinson emphasized this diagnostic role of error analysis, describing its
purpose as revealing why an algorithm is unstable and thereby ``to expose
what sort of change is necessary to make it stable''
\cite{Wilkinson1974}.
Recent analyses provide concrete examples of this principle.
In Householder orthogonalization with a nonstandard inner product, Shao's
analysis isolates the normalization of the Householder vectors as the
critical source governing the loss of orthogonality~\cite{Shao2023}.
In the backward error analysis of $s$-step GMRES, Carson and Ma separate the
effects of orthogonalization and the least-squares solve, making explicit how
the orthogonalization process influences the final backward error
\cite{CM2025}.
More recently, Burke, Carson, and Ma used a modular GMRES framework
\cite{BHMV2024} to analyze the numerical stability of sketched GMRES,
showing how different algorithmic components can be analyzed and replaced
within a common structural framework~\cite{BCM2026}.
These examples show that preserving error provenance while separating
structural propagation from kernel-specific bounds can make error analysis a
tool not only for certifying stability, but also for diagnosing, modifying,
and reusing numerical algorithms.

The same changes in arithmetic that motivate trackable analysis also make the
dimension dependence of rounding error bounds increasingly important in
practice.
In particular, the widespread use of low-precision hardware and the
proliferation of low-precision formats mean that the influence of the problem
dimension can no longer routinely be regarded as small.
A large part of modern rounding error analysis is expressed using Higham's
widely adopted \(\gammanh\)-notation~\cite{Higham2002}, that is,
\begin{equation}
\label{eq:bound1}
\gammanh=\frac{n\macheps}{1-n\macheps},
\end{equation}
which requires the assumption \(n\macheps<1\).
The restriction \(n\macheps<1\) becomes increasingly relevant when large
problem dimensions are combined with low-precision arithmetic.
When \(n\macheps\geq 1\), bounds expressed through \(\gammanh\) are no longer
applicable, even though the underlying computation and its rounding errors
remain well defined.
Thus, the standard notation can cease to provide a useful error bound exactly
in regimes that are becoming increasingly common in low-precision computing.

It is worth noting that Wilkinson's original
framework~\cite{Wilkinson1963} does not require the assumption
\(n\macheps<1\).
Its basic quantity is
\begin{equation}
\label{eq:bound0}
\gammanw=(1+\macheps)^n-1.
\end{equation}
This notation can describe the accumulation of multiplicative rounding errors
without any restriction on \(n\macheps\).
However, it does not provide an equally convenient mechanism for absorbing
rounding errors introduced by division.
In particular, rounding error analysis involving division naturally produces
factors such as \(1/(1-\macheps)\).
These factors cannot be absorbed into Wilkinson's framework simply by
increasing the exponent of \((1+\macheps)\), since
\(1/(1-\macheps)>1+\macheps\) by one, as demonstrated in
Example~\ref{example:div} in Section~\ref{sec:model}.
Consequently, analyses that involve division can lead to increasingly
cumbersome expressions when written solely in terms of \(\gammanw\).

To remove the restriction \(n\macheps<1\) while retaining a convenient way to
treat division-induced errors, we introduce a new notation
\begin{equation}
\label{eq:bound}
\gamman=(1-\macheps)^{-n}-1.
\end{equation}
Whenever \(n\macheps<1\), this quantity satisfies
\(\gammanw<\gamman\leq\gammanh\), with equality in the second inequality
only for \(n=1\).
In~\cite{O1996}, O'Cinneide adopted \((1-\macheps)^{-n}\) as an upper bound 
on products of rounding error factors, but did not introduce it as part of a 
general notation for rounding error analysis.
This notation \(\gamman\) was first used in our earlier work~\cite{Sai2022}, 
where we demonstrated that it offers an effective framework for rounding 
error analysis.
The proposed notation preserves Wilkinson's ability to describe accumulated
rounding errors without requiring \(n\macheps<1\), while retaining much of the
algebraic simplicity that has made the \(\gammanh\)-notation attractive in
practical error-bound derivations.
In particular, it provides a convenient mechanism for treating
rounding-error factors induced by division and facilitates the derivation of
compact error bounds for both standard- and mixed-precision algorithms.

The new notation and the trackable rounding error analysis framework play
complementary roles in the methodology developed here.
The former provides a convenient way to bound accumulated rounding errors
without the restriction \(n\macheps<1\).
The latter is trackable because it preserves the source and propagation of
individual errors, and modular because the corresponding local bounds can be
replaced without changing the structural derivation.
This separation makes it possible to reuse the same analysis under different
arithmetic models, precision assignments, and implementations, while also
revealing how changes to individual computational kernels affect the final
result.

The remainder of this paper is organized as follows.
In Section~\ref{sec:model}, we present the proposed notation and establish its
fundamental properties.
In Section~3, we illustrate the trackable rounding error analysis framework
using LU factorization as a representative example.

\section{Basic Rounding Model}
\label{sec:model}

In order to perform a rounding error analysis of an algorithm, we need to 
make some assumptions regarding the accuracy of the basic arithmetic 
operations.
There are two standard models:
\begin{equation}
\label{eq:model0}
\fl(\alpha\op\beta)=(\alpha\op\beta)(1+\delta),
\qquad \lvert\delta\rvert\leq\macheps,
\end{equation}
and
\begin{equation}
\label{eq:model1}
\fl(\alpha\op\beta)=\frac{\alpha\op\beta}{1+\delta},
\qquad \lvert\delta\rvert\leq\macheps,
\end{equation}
for any finite floating-point numbers \(\alpha\), \(\beta\in\mathbb R\) and
\(\op\in\set{+,-,\times,/}\).

These standard models provide a simple and effective description of the effect 
of a single floating-point operation.
However, practical numerical algorithms typically involve a large number of 
arithmetic operations, each introducing its own rounding error. 
Consequently, the direct application of the standard model often leads to 
expressions containing many factors of the form \((1+\delta_i)\), where each 
\(\delta_i\) satisfies \(\abs{\delta_i}\le \macheps\).
As the number of operations increases, such expressions quickly become cumbersome and difficult to manipulate.
For example, after \(n\) successive floating-point operations, one 
frequently encounters products of the form
\[
(1+\delta_1)(1+\delta_2)\cdots(1+\delta_n).
\]
For both Models~\eqref{eq:model0} and~\eqref{eq:model1}, this product can be
expressed in the form \(1+\theta_n\), where
\[
(1+\delta_1)(1+\delta_2)\cdots(1+\delta_n)=1+\theta_n,
\]
with \(\abs{\theta_n}\leq (1+u)^n-1=\gammanw<(1-u)^{-n}-1=\gamman\leq n\macheps/(1-n\macheps)= \gammanh\)
for Model~\eqref{eq:model0}, and
\[
\frac{1}{(1+\delta_1)(1+\delta_2)\cdots(1+\delta_n)}=1+\theta_n,
\]
with \(\abs{\theta_n}\leq (1-u)^{-n}-1=\gamman< n\macheps/(1-n\macheps)= \gammanh\)
for Model~\eqref{eq:model1}.
It is important to note that the upper bound on \(\abs{\theta_n}\) associated with \(\gammanh\) requires \(n\macheps<1\).
Observe that, for Model~\eqref{eq:model1}, the inequality
\(\abs{\theta_n}\leq (1-u)^{-n}-1=\gamman\) is sharp, so any bound strictly smaller than \(\gamman\), such as \(\gammanw\), is not valid.
Consequently, Wilkinson's analysis applies only to Model~\eqref{eq:model0}, 
whereas our \(\gamman\)-based bound and Higham's analysis are applicable to 
both Models~\eqref{eq:model0} and~\eqref{eq:model1}.
Besides \(\gammanw\) and \(\gammanh\), another reasonable notation is
 \(\exp(n\macheps)-1\), which satisfies
\[
(1+\macheps)^n-1
<\exp(n\macheps)-1
<(1-\macheps)^{-n}-1
<\frac{n\macheps}{1-n\macheps}.
\]
Therefore, similarly to \(\gammanw\), it cannot work for Model~\eqref{eq:model1}.
In Figure~\ref{fig:comparison-gamma}, we compare these notations.

We then state the following lemma to show that
\(\gamman=(1-\macheps)^{-n}-1\) is very close to Wilkinson's notation, i.e., 
\(\gammanw=(1+\macheps)^n-1\).

\begin{figure}[!tb]
\label{fig:comparison-gamma}
\caption{Comparison of different notations.}
\begin{tabular}{cc}
\includegraphics[width=0.45\textwidth]{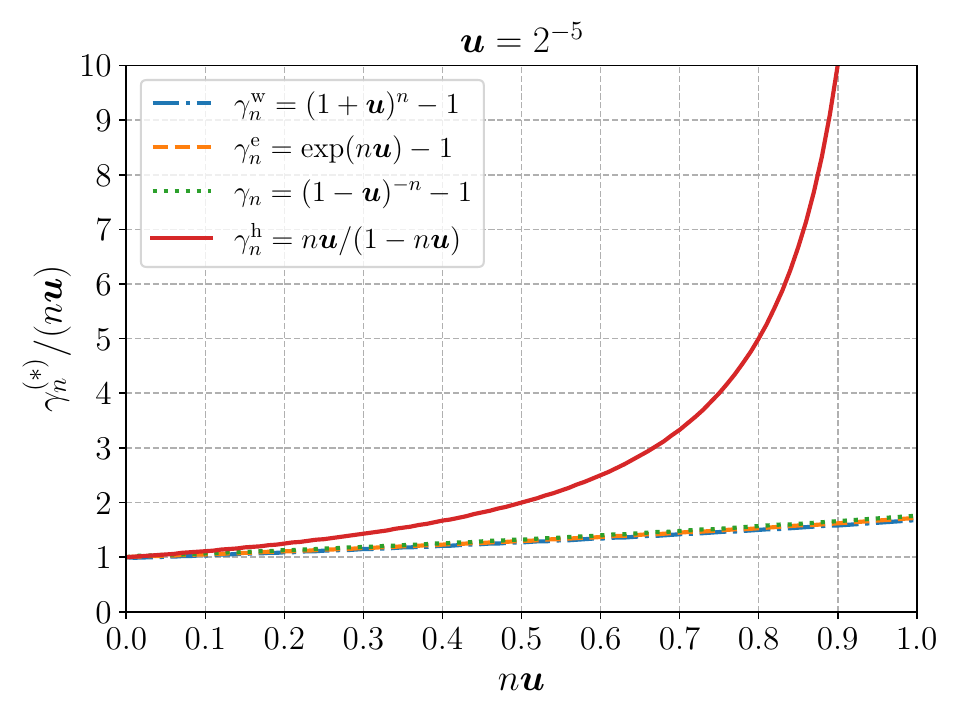} &
\includegraphics[width=0.45\textwidth]{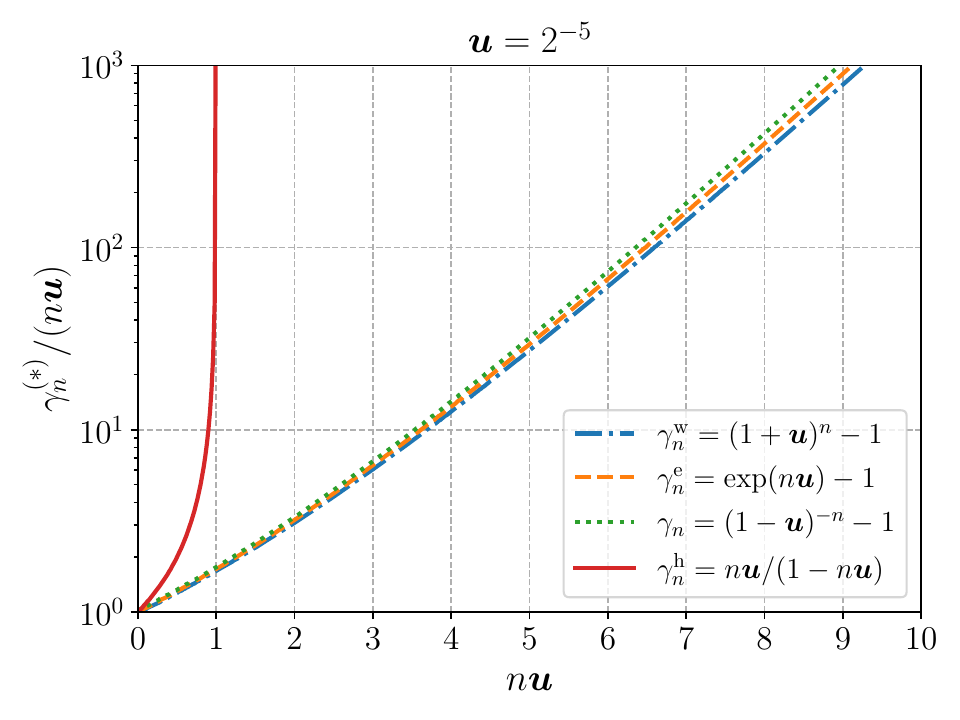} \\
(a) & (b)
\end{tabular}
\end{figure}

\begin{lemma}
Let \(\macheps\in(0,1)\) and \(n\) be a positive integer.
Then
\[
\frac{\gamman-\gammanw}{\gammanw}\leq \macheps(1-\macheps)^{-n}.
\]
\end{lemma}


\begin{proof}
Let \(\alpha=(1-\macheps)^{-1}\) and \(\beta=1+\macheps\).
Noticing  that
\[
\alpha^{-1}\beta = 1-\macheps^2\leq 1,
\]
we then have
\begin{equation*}
\gamman-\gammanw
=\alpha^n-\beta^n
=(\alpha-\beta)\sum_{i=0}^{n-1}\alpha^{n-1-i}\beta^i
\leq n(\alpha-\beta)\alpha^{n-1}
=n\macheps^2(1-\macheps)^{-n},
\end{equation*}
which concludes the proof by combining this bound with \(\gammanw\geq n\macheps\).
\end{proof}



We then demonstrate that the quantity represented by the proposed notation
\(\gamman\) is still of order \(n\macheps\), consistent with the classical
\(\gammanh\)-notation.

\begin{theorem}
If \(\macheps\leq1/6\) and \(n\macheps\leq1\), then
\[
(1-\macheps)^{-n}-1<2n\macheps.
\]
\end{theorem}

\begin{proof}
By the concavity of \(\log(1-x)\), for \(\macheps\leq 1/6\),
we have
\[
\log(1-\macheps)\geq 6\macheps\log\Bigl(\frac{5}{6}\Bigr),
\]
which implies that
\[
(1-\macheps)^{-n}\leq\Bigl(\frac{6}{5}\Bigr)^{6n\macheps}<3^{n\macheps}.
\]
At the same time, noticing that
\[
3^{n\macheps}\leq (1-n\macheps)\cdot 3^0+n\macheps\cdot 3^1
=1+2n\macheps
\]
by \(n\macheps\leq 1\), we can conclude the proof.

\end{proof}

Similarly to~\cite[Lemmas~3.1 and~3.3]{Higham2002}, the next lemma provides 
the necessary properties for \(\gamman\) defined in~\eqref{eq:bound}.

\begin{lemma} \label{lem:gamma_property}
If \(\abs{\delta_i}\leq\macheps\) and \(\rho_i=\pm 1\) for \(i=1\),
\(\dotsc\), \(n\), then
\begin{equation} \label{eq:lem-property:0}
\prod_{i=1}^n(1+\delta_i)^{\rho_i}=1+\theta_n, \qquad\abs{\theta_n}\leq \gamman.
\end{equation}
For any positive integer \(k\), \(i\), \(j\), and \(0<\macheps<1\), the following 
relations hold:
\begin{align}
\gamma_k\gamma_j&\leq\gamma_{\min(k,j)}\quad\text{for}\quad \macheps\leq 1-2^{-1/\max(k,j)}, \label{eq:lem-property:1}\\
i\gamma_k&\leq \gamma_{ik}, \label{eq:lem-property:2} \\
\gamma_k+\macheps&\leq\gamma_{k+1}, \label{eq:lem-property:3} \\
\gamma_k+\gamma_j+\gamma_k\gamma_j&=\gamma_{k+j}. \label{eq:lem-property:4}
\end{align}
\end{lemma}

\begin{proof}
The first conclusion~\eqref{eq:lem-property:0} is obtained by noting that
\(1+\delta_i\leq 1/(1-\macheps)\) and \((1+\delta_i)^{-1}\leq 1/(1-\macheps)\).
From the assumption \(0<\macheps<1\) and \(\macheps\leq 1-2^{-1/\max(k,j)}\),
we have \(0<\gamma_k\leq 1\) and
\(0<\gamma_j\leq 1\), which directly proves~\eqref{eq:lem-property:1}.
From the definition in~\eqref{eq:bound}, \eqref{eq:lem-property:4} can be verified straightforwardly.
We then have~\eqref{eq:lem-property:2} by repeatedly accumulating \(\gamma_k\) in~\eqref{eq:lem-property:4}.
Using \(\macheps\leq\gamma_1\) in~\eqref{eq:lem-property:4} with \(j=1\), we obtain~\eqref{eq:lem-property:3}.
\end{proof}

\begin{remark}
In comparison with~\cite[Lemma~3.3]{Higham2002}, we emphasize 
that~\eqref{eq:lem-property:4} holds as an equality for \(\gamman\), rather 
than as an inequality for \(\gammanh\).
\end{remark}

We next illustrate the effectiveness of the proposed notation through the 
following two examples.

\begin{example} \label{example:div}
Forward error in \(1/(\alpha+\beta)\).
\end{example}

This example illustrates that our proposed framework \(\gamman\) can treat 
the division cases that \(\gammanw\) and \(\exp(n\macheps)-1\) are unable to 
address, as discussed in Section~\ref{sec:introduction}.

Under Model~\eqref{eq:model0}, the computed result of \(1/(\alpha+\beta)\) satisfies
\begin{equation}
    \fl\Bigl(\frac{1}{\alpha+\beta}\Bigr)
    = \frac{1}{\bigl(\alpha+\beta\bigr)(1+\delta_1)}(1+\delta_2),
    \qquad \lvert\delta_1\rvert\leq\macheps, 
    \qquad \lvert\delta_2\rvert\leq\macheps,
\end{equation}
where \(\delta_1\) and \(\delta_2\) come from addition and division,
respectively.
Thus, the forward error in \(1/(\alpha+\beta)\) can be bounded as
follows:
\begin{equation*}
    \begin{split}
        \absbig{\fl\Bigl(\frac{1}{\alpha+\beta}\Bigr)-\frac{1}{\alpha+\beta}}
        &= \absbig{\frac{1}{\alpha+\beta}}
        \absbig{\frac{1+\delta_2}{1+\delta_1}-1}
        \leq \bigl((1-\macheps)^{-2}-1\bigr)\cdot
        \absbig{\frac{1}{\alpha+\beta}}.
    \end{split}
\end{equation*}
Even when we adopt Model~\eqref{eq:model0}, we still have to deal with the term \(1/(1+\delta_1)\).
As noted at the beginning of this section for Model~\eqref{eq:model1}, the 
bound \((1-\macheps)^{-2}-1\) is tight for this expression, and therefore any 
bound that is smaller than \((1-\macheps)^{-2}-1\), such as \((1+\macheps)^2-1\)
or \(\exp(2\macheps)-1\), fails to cover this situation.

\begin{example}
Error analysis for harmonic series \(\sum_{n=1}^{\infty}1/n\).
\end{example}

Let \(H_k:=\sum_{n=1}^{k}1/n\).
We compute \(H_k\) by forward summation, i.e., the terms are accumulated in the order
\[
1,\ \frac12,\ \frac13,\ \ldots,\ \frac1k.
\]
From~\eqref{eq:model0}, the computed result \(\hat H_{k}\) satisfies
\begin{equation}
\begin{split}
    \hat H_{k}
    &=H_k+\sum_{n=1}^k \frac{1}{n}\underbrace{\Bigl((1+\delta_{n}^{(\mathtt{div})})
    \prod_{i=\max(n,2)}^k(1+\delta_{i}^{(\mathtt{add})})-1\Bigr)}_{=:\delta_n},
\end{split}
\end{equation}
where \(\delta_n\) satisfies \(\abs{\delta_n}\leq (1+\macheps)^k-1\).
This means that
\[
\hat H_{k} = H_{k}(1+\Delta H_k)
\]
with \(\abs{\Delta H_k}\leq (1+\macheps)^k-1 \leq (1-\macheps)^{-k}-1\leq\gamma^h_k\), where only the last inequality requires \(k\macheps<1\).
Although the Wilkinson-type bound \((1+\macheps)^k-1\) is the sharpest of these three
bounds for this particular example, the proposed \(\gamma_k\) is still no larger
than the conventional \(\gamma_k^h\) whenever \(k\macheps<1\).
More importantly, unlike \(\gamma_k^h\), the proposed bound remains well defined and
valid when \(k\macheps\ge1\).

The forward summation of the harmonic series also provides a simple illustration
of stagnation in finite-precision arithmetic.
As the partial sum grows while the terms \(1/n\) decrease, eventually \(1/n\)
becomes too small to change the current floating-point partial sum.
Under IEEE round-to-nearest arithmetic, this occurs at \(n=2^{21}\) in single precision
and \(n=2^{48}+1\) in double precision~\cite[Table~6.3]{HP2019}.
Thus, in these formats the forward summation stagnates even before \(n\macheps\) reaches one.
Nevertheless, if the computation is formally continued beyond this point, the proposed
\(\gamma_n\) bound remains applicable even for \(n\macheps\ge1\), whereas the
conventional \(\gammanh\) bound is no longer available.

\section{Trackable and Modular Rounding Error Analysis}
\label{sec:trackable}

In this section, we illustrate the proposed trackable and modular rounding error 
analysis framework through a representative example, i.e., LU factorization. 
Unlike classical analysis that focuses primarily on deriving compact global 
error bounds, the proposed approach retains information about the origin and 
propagation of individual rounding errors throughout the computation.
As a result, it provides a more detailed characterization of the error 
structure and makes it possible to assess the contribution of different 
computational components to the final error.
This information is particularly useful for comparing algorithmic variants 
and for identifying operations that are most sensitive to precision 
reduction.
The following analysis demonstrates how these features can be exploited in practice.

Before proceeding, we first present the assumptions regarding basic
operations on vectors and matrices in Table~\ref{tab:notation}.
In the notation \(\Delta(X)^{(\mathtt{op})}(n)\), \(X\) indicates the quantity 
with which the error is associated, the superscript \(\mathtt{op}\) identifies 
the operation that produces the error, and the argument \(n\) indicates its 
dependence on the problem dimension. This notation allows us to track both the 
source and the propagation of individual rounding errors throughout the computation.
The assumptions in Table~\ref{tab:notation} are also modular: by assigning
different bounds to local perturbations, the same structural analysis can
accommodate, for example, different rounding models such as~\eqref{eq:model0} 
and~\eqref{eq:model1}, mixed precision arithmetic, compensated kernels, and 
Ozaki-style matrix multiplication.

\begin{table}[!tb]
\centering
\caption{Assumptions for basic operations on vectors and matrices.
Here \(\alpha\), \(\beta\in\mathbb{F}\) are scalars;
\(x\), \(y\in\mathbb{F}^{n}\) are vectors; \(A\in\mathbb F^{m\times n}\),
\(B\in\mathbb F^{n\times k}\), \(C\in\mathbb F^{m\times k}\), and
\(Y\in\mathbb F^{n\times m}\) are matrices;
\(U\in\mathbb F^{n\times n}\) is a triangular matrix.
}
\label{tab:notation}
\begin{spacing}{1.3}
\medskip
\begin{tabular}{ccc}
\hline Operation & \begin{tabular}[c]{@{}c@{}}BLAS/LAPACK\\ Subroutine\end{tabular} & Assumption\\ \hline
addition & - & \(\fl(\alpha\pm\beta)=(\alpha\pm\beta)(1+\delta^{(\mathtt{add})})\)\\
\hline
multiplication & - & \(\fl(\alpha\beta)=\alpha\beta(1+\delta^{(\mathtt{mult})})\)\\
\hline
division & - & \(\fl\bigl(\alpha\beta^{-1}\bigr)=\alpha\beta^{-1}\,(1+\delta^{(\mathtt{div})})\)\\
\hline
dot product  & \texttt{XDOT}/\texttt{XDOTC} & \(\fl(x\herm y) = (x+\Delta x^{(\mathtt{dot})}(n))\herm y\)\\
\hline
\begin{tabular}[c]{@{}c@{}}matrix--matrix\\ multiplication\end{tabular}  & \texttt{XGEMM} & \(AB+\alpha C = \hat C + \Delta(AB)^{(\mathtt{gemm})}(n) + (\Delta C)^{(\mathtt{gemm})}\)\\
\hline
triangular solver & \texttt{XTRSM}/\texttt{XTRSV} & \(U\hat X=Y+\Delta(U\hat X)^{(\mathtt{trsv})}(n)\)\\
\hline
\begin{tabular}[c]{@{}c@{}}non-block\\ LU factorization\end{tabular} & \texttt{XGETF2} & \(\hat A + \Delta(LU)^{(\mathtt{getf2})}(n) = \hat L\hat U\)\\ \hline
\end{tabular}
\end{spacing}
\end{table}

In Section~\ref{example:lu}, we present the trackable rounding error analysis 
for the block LU factorization through an illustrative example, which is 
connected to the triangular system solver analyzed in Section~\ref{example:triangular-solver}.
From the following analysis, the main source of error in the LU factorization
becomes evident.
In addition, we use these examples to highlight that the main sources of 
errors can vary across different algorithms.

\subsection{Trackable error analysis for triangular systems}\label{example:triangular-solver}
There are essentially two approaches to implementing a solver for upper triangular systems:
one is based on vector dot products, i.e., \texttt{XDOT}/\texttt{XDOTC}, 
and the other is based on vector additions, i.e., \texttt{XAXPY}.
These two approaches for upper triangular systems are described in Algorithms~\ref{alg:trsv-dot} 
and~\ref{alg:trsv-axpy}, respectively.
In the following lemmas, we analyze these algorithms based on the assumptions 
shown in Table~\ref{tab:notation}.
The corresponding algorithms and analysis for lower triangular systems are 
analogous.

\begin{algorithm}[!tb]
\caption{The \texttt{XDOT}-based upper triangular system solver}
\label{alg:trsv-dot}
\begin{algorithmic}[1]
\REQUIRE An upper triangular matrix \(U\in\mathbb F^{n\times n}\),
a right-hand side \(b\in\mathbb F^{n}\).
\ENSURE The solution \(x\in\mathbb F^{n}\) satisfying \(Ux = b\).

\STATE \(x_n = U_{nn}^{-1}b_n\)
\STATE \(s = b\)
\FOR{\(i = n-1:-1:1\)}
    \STATE \(s_i = b_i - U_{i,i+1:n} x_{i+1:n}\) \label{line:trsv-dot:dot}
    \STATE \(x_i = U_{i,i}^{-1}s_{i}\)
\ENDFOR
\end{algorithmic}
\end{algorithm}

\begin{algorithm}[!tb]
\caption{The \texttt{XAXPY}-based upper triangular system solver}
\label{alg:trsv-axpy}
\begin{algorithmic}[1]
\REQUIRE An upper triangular matrix \(U\in\mathbb F^{n\times n}\),
a right-hand side \(b\in\mathbb F^{n}\).
\ENSURE The solution \(x\in\mathbb F^{n}\) satisfying \(Ux = b\).

\STATE \(x_n = U_{nn}^{-1}b_n\)
\STATE \(s = b\)
\FOR{\(i = n-1:-1:1\)}
    \STATE \(s_{1:i} = s_{1:i} - U_{1:i,i+1} x_{i+1}\)
    \label{line:trsv-axpy:axpy}
    \STATE \(x_i = U_{i,i}^{-1}s_{i}\)
\ENDFOR
\end{algorithmic}
\end{algorithm}

\begin{lemma}[The \texttt{XDOT}-based algorithm] \label{lem:triangular-solver-xdot}
    Assume that \(U\in\mathbb F^{n\times n}\) is a nonsingular upper triangular
    matrix and \(b\in\mathbb F^n\).
    Suppose that the triangular system \(Ux=b\) is solved by the \texttt{XDOT}-based algorithm (i.e., Algorithm~\ref{alg:trsv-dot}) with
    \begin{align}
        \hat x_i&=U_{i,i}^{-1}(1+\delta_i^{(\mathtt{div})})\hat s_i,
        \label{eq:trsv-dot:div}\\
        \fl(U_{i,i+1:n}\hat x_{i+1:n})&=\bigl(U_{i,i+1:n}+\Delta U^{(\mathtt{dot})}_{i,i+1:n}(n-i)\bigr)\hat x_{i+1:n},\qquad i<n, \label{eq:trsv:dot}\\
        \hat s_i&=(1+\delta_i^{(\mathtt{add})})
        \left(b_i-\fl(U_{i,i+1:n}\hat x_{i+1:n})\right),\qquad i<n.
        \label{eq:trsv:sub}
    \end{align}
    where \(\hat s_n=b_n\) and \(\delta_n^{(\mathtt{add})}:=0\).
    If \(\abs{\delta_i^{(\mathtt{add})}}<1\) and
    \(\abs{\delta_i^{(\mathtt{div})}}<1\), then there exists an upper triangular matrix \(\Delta U^{(\mathtt{trsv})}(n)\) such that the computed solution \(\hat x\) satisfies
    \begin{equation} \label{eq:lem-trsv-dot}
        \bigl(U + \Delta U^{(\mathtt{trsv})}(n)\bigr) \hat x = b,
    \end{equation}
    where
    \begin{equation} \label{eq:lem-trsv-dot-trackable}
        \Delta U^{(\mathtt{trsv})}_{i,j}(n)=
        \begin{cases}
        \Delta U^{(\mathtt{dot})}_{i,j}(n-i), & i<j,\\[1mm]
        \displaystyle\left((1+\delta_i^{(\mathtt{add})})^{-1}
        (1+\delta_i^{(\mathtt{div})})^{-1}-1\right)U_{i,i}, & i=j.
        \end{cases}
    \end{equation}
\end{lemma}
\begin{proof}
    See Appendix~\ref{subsec:triangular-xdot}.
\end{proof}

\begin{lemma}[The \texttt{XAXPY}-based algorithm] \label{lem:triangular-solver-xaxpy}
    Assume that \(U\in\mathbb F^{n\times n}\) is a nonsingular upper triangular matrix and \(b\in\mathbb F^n\).
    Suppose that \(Ux=b\) is solved by the \texttt{XAXPY}-based algorithm (i.e., Algorithm~\ref{alg:trsv-axpy}) with~\eqref{eq:trsv-dot:div} and
    \begin{align}
        \fl(\hat s_{k} - U_{k,i+1} \hat x_{i+1})
        &= (1+\delta_{k,i+1}^{(\mathtt{add})})
        \left(\hat s_k-(1+\delta_{k,i+1}^{(\mathtt{mult})})U_{k,i+1}\hat x_{i+1}\right)
        \label{eq:trsv:axpy}
    \end{align}
    for \(1\leq k\leq i\) and \(1 \leq i \leq n-1\).
    If \(\abs{\delta_{i,j}^{(\mathtt{add})}}<1\) and
    \(\abs{\delta_{i}^{(\mathtt{div})}}<1\), then there exists an upper triangular matrix \(\Delta U^{(\mathtt{trsv})}(n)\) such that the computed solution \(\hat x\) satisfies~\eqref{eq:lem-trsv-dot},
    where
    \begin{equation} \label{eq:lem-trsv-axpy-trackable}
        \Delta U^{(\mathtt{trsv})}_{i,j}(n)=
        \begin{cases}
        \displaystyle\bigl((1+\delta_{i,n}^{(\mathtt{add})})^{-1}
        \cdots(1+\delta_{i,j+1}^{(\mathtt{add})})^{-1}
        (1+\delta_{ij}^{(\mathtt{mult})})-1\bigr)U_{i,j}, & i<j,\\[3mm]
        \displaystyle\bigl((1+\delta_{i,n}^{(\mathtt{add})})^{-1}
        \cdots(1+\delta_{i,i+1}^{(\mathtt{add})})^{-1}
        (1+\delta_i^{(\mathtt{div})})^{-1}-1\bigr)U_{i,i}, & i=j.
        \end{cases}
    \end{equation}
\end{lemma}

\begin{proof}
    See Appendix~\ref{subsec:triangular-xaxpy}.
\end{proof}

For a single right-hand side, these lemmas both give the residual form
\begin{equation}
U\hat x=b+\Delta(U\hat x)^{(\mathtt{trsv})}(n),\qquad
\abs{\Delta(U\hat x)^{(\mathtt{trsv})}(n)}\leq\epstrsv(n)\abs U\abs{\hat x},
\label{eq:trsv-residual-vector}
\end{equation}
whenever \(\abs{\Delta U^{(\mathtt{trsv})}(n)}\leq\epstrsv(n)\abs U\), since
\(\Delta(U\hat x)^{(\mathtt{trsv})}(n)=-\Delta U^{(\mathtt{trsv})}(n)\hat x\).
Applying this result independently to the columns of a multiple right-hand side
solve gives the \texttt{XTRSM} model in Table~\ref{tab:notation}:
\begin{equation}
U\hat X=Y+\Delta(U\hat X)^{(\mathtt{trsv})}(n),\qquad
\abs{\Delta(U\hat X)^{(\mathtt{trsv})}(n)}\leq\epstrsv(n)\abs U\abs{\hat X}.
\label{eq:trsm-residual-matrix}
\end{equation}
The columns may have different coefficient perturbations and this
model requires no common coefficient perturbation.

Lemmas~\ref{lem:triangular-solver-xdot} and~\ref{lem:triangular-solver-xaxpy}
show that the two variants have different error structures.
For the \texttt{XDOT}-based algorithm, the off-diagonal perturbations are 
inherited directly from the dot products, whereas for the
\texttt{XAXPY}-based algorithm they involve accumulated multiplication and
addition errors.
Hence, a more accurate dot product can directly improve the former, but not the latter; for example, see Section~\ref{remark:Kahan} for further discussion.


\subsection{Trackable error analysis for block LU factorization}\label{example:lu}
In this example, we consider two different implementations of the block LU 
factorization, namely Algorithms~\ref{alg:block-LU} and~\ref{alg:block-LU-2}.
The differences between these two algorithms occurs in 
Lines~\ref{line:lu:tr-solver} and~\ref{line:update_trail}.
For simplicity, in both Algorithms~\ref{alg:block-LU} and~\ref{alg:block-LU-2},
we assume that \(m = qn_b\) and \(n = pn_b\), and the displayed algorithms 
omit pivoting.
In the following theorems, we analyze these algorithms based on the assumptions 
shown in Table~\ref{tab:notation}.

\begin{algorithm}[!tb]
\caption{The right-looking block LU factorization}
\label{alg:block-LU}
\begin{algorithmic}[1]
\REQUIRE A block size \(n_b\), and a given matrix \(\Acal\in\mathbb F^{m\times n}\) 
with \(m\geq n\),  \(m = qn_b\), and \(n = pn_b\).
\ENSURE A lower trapezoidal matrix \(\LL\in\mathbb F^{m\times n}\)
and an upper triangular matrix \(\UU\in\mathbb F^{n\times n}\)
satisfying \(\Acal = \LL\UU\).

\FOR{\(i = 1:p-1\)}
    \STATE Compute the LU factorization \(\Acal_{i:q, i} = \LL_{i:q, i}
    \UU_{i,i}\). \label{line:lu:sub-lu}
    \STATE Update \(\UU_{i, i+1:p}\) by solving the lower triangular 
    system \(\LL_{i,i} \UU_{i, i+1:p} = \Acal_{i, i+1:p}\). \label{line:lu:tr-solver}
    \STATE Update trailing submatrix by \(\Acal_{i+1:q, i+1:p} = 
    \Acal_{i+1:q, i+1:p} - \LL_{i+1:q, i} \UU_{i, i+1:p}\). \label{line:update_trail}
\ENDFOR
\STATE Compute the LU factorization \(\Acal_{p:q, p} = \LL_{p:q, p}
\UU_{p,p}\).
\end{algorithmic}
\end{algorithm}

\begin{algorithm}[!tb]
\caption{The left-looking block LU factorization}
\label{alg:block-LU-2}
\begin{algorithmic}[1]
\REQUIRE A block size \(n_b\), and a given matrix \(\Acal\in\mathbb F^{m\times n}\) 
with \(m\geq n\),  \(m = qn_b\), and \(n = pn_b\).
\ENSURE A lower trapezoidal matrix \(\LL\in\mathbb F^{m\times n}\)
and an upper triangular matrix \(\UU\in\mathbb F^{n\times n}\)
satisfying \(\Acal = \LL\UU\).

\FOR{\(i = 1:p-1\)}
    \STATE Compute the LU factorization \(\Acal_{i:q, i} = \LL_{i:q, i}
    \UU_{i,i}\).
    \STATE Update \(\UU_{i, i+1:p}\) by solving the lower triangular 
    system \(\LL_{i,i} \UU_{i, i+1:p} = \Acal_{i, i+1:p} - \LL_{i,1:i-1} \UU_{1:i-1, i+1:p}\).
    \STATE Update trailing submatrix by \(\Acal_{i+1:q, i+1} = 
    \Acal_{i+1:q, i+1} - \LL_{i+1:q, 1:i} \UU_{1:i, i+1}\).
\ENDFOR
\STATE Compute the LU factorization \(\Acal_{p:q, p} = \LL_{p:q, p}
\UU_{p,p}\).
\end{algorithmic}
\end{algorithm}

\begin{theorem}[The right-looking block LU factorization] \label{thm:block-lu-right}
    Let \(\LL\in\mathbb{F}^{m\times n}\) and \(\UU\in\mathbb{F}^{n\times n}\)
    denote a lower trapezoidal matrix and an upper triangular matrix, respectively.
    Assume that \(\Acal = \LL\UU\) is computed by
    Algorithm~\ref{alg:block-LU}.
    Also, assume that Lines~\ref{line:lu:sub-lu}--\ref{line:update_trail} in
    Algorithm~\ref{alg:block-LU} satisfy, respectively,
    \begin{align}
        \Acalhat^{(i-1)}_{i:q,i}
        + \Delta\bigl(\LL_{i:q,i}\UU_{i,i}\bigr)^{(\mathtt{getf2})}(n_b)
        &= \LLhat_{i:q,i}\UUhat_{i,i},
        \label{eq:lem4:blocklu:trf2} \\
        \LLhat_{i,i}\UUhat_{i,i+1:p}
        &= \Acalhat^{(i-1)}_{i,i+1:p}+\Delta(\LL_{i,i}\UU_{i,i+1:p})^{(\mathtt{trsv})}(n_b),
        \label{eq:lem4:blocklu:trsm} \\
        \Acalhat^{(i-1)}_{i+1:q,i+1:p}
        - \LLhat_{i+1:q,i}\UUhat_{i,i+1:p}
        &= \Acalhat^{(i)}_{i+1:q,i+1:p}
        + \bigl(\Delta\bigl(\LL_{i+1:q,i}\UU_{i,i+1:p}\bigr)^{(\mathtt{gemm})}(n_b)\bigr)^{(i)}
        \notag\\
        &\quad
        + \bigl(\Delta\Acal^{(i-1)}_{i+1:q,i+1:p}\bigr)^{(\mathtt{gemm})},
        \label{eq:lem4:blocklu:update-trail}
    \end{align}
    where \(\Acalhat^{(0)}_{i+1:q,i+1:p}:=\Acal_{i+1:q,i+1:p}\).
    Then there exists \(\Delta(\LL\UU)^{(\mathtt{getrf})}(n,n_b)\) such that
    \begin{equation} \label{eq:blocklu:getrf}
        \Acal + \Delta(\LL\UU)^{(\mathtt{getrf})}(n,n_b) = \LLhat\UUhat,
    \end{equation}
    where, for the \(i\)th block column with \(i=1,\ldots,p\),
    \begin{equation} \label{eq:blocklu:getrf-trackable}
    \begin{split}
        \Delta(\LL\UU)^{(\mathtt{getrf})}_{i:q,i}(n,n_b)
        ={}&-\sum_{k=1}^{i-1}
        \bigl(\Delta(\LL_{i:q,k}\UU_{k,i})^{(\mathtt{gemm})}(n_b)\bigr)^{(k)}
        \\
        &-\sum_{k=0}^{i-2}
        \bigl(\Delta\Acal^{(k)}_{i:q,i}\bigr)^{(\mathtt{gemm})}
        + \Delta(\LL_{i:q,i}\UU_{i,i})^{(\mathtt{getf2})}(n_b),
    \end{split}
    \end{equation}
    and, for the off-diagonal blocks in the \(i\)th block row with \(i=1,\ldots,p-1\),
    \begin{equation} \label{eq:blocklu:getrf-trsv}
    \begin{split}
        \Delta(\LL\UU)^{(\mathtt{getrf})}_{i,i+1:p}(n,n_b)
        ={}&-\sum_{k=1}^{i-1}\bigl(\Delta(\LL_{i,k}\UU_{k,i+1:p})^{(\mathtt{gemm})}(n_b)\bigr)^{(k)}
        -\sum_{k=0}^{i-2}\bigl(\Delta\Acal^{(k)}_{i,i+1:p}\bigr)^{(\mathtt{gemm})} \\
        &+\Delta(\LL_{i,i}\UU_{i,i+1:p})^{(\mathtt{trsv})}(n_b).
    \end{split}
    \end{equation}

    For the case \(n_b=1\), Algorithm~\ref{alg:block-LU} reduces to a non-block
    LU factorization.
    The \texttt{XGETF2} and \texttt{XTRSV} operations in~\eqref{eq:lem4:blocklu:trf2}
    and~\eqref{eq:lem4:blocklu:trsm}, respectively, reduce to divisions, 
    i.e.,~\eqref{eq:lem4:blocklu:trf2} reduces to
    \begin{equation} \label{eq:blocklu:div}
        \LLhat_{j,i} = (1+\delta^{(\mathtt{div})}_{j,i})\Acalhat^{(i-1)}_{j,i}\UUhat_{i,i}^{-1},
    \end{equation}
    and \(\Delta(\LL_{i,i}\UU_{i,i+1:p})^{(\mathtt{trsv})}(1)=0\) due to \(\LLhat_{i,i}=1\).
    Then there exists \(\Delta(\LL\UU)^{(\mathtt{getf2})}(n)\) such that
    \begin{equation} \label{eq:blocklu:getf2}
        \Acal + \Delta(\LL\UU)^{(\mathtt{getf2})}(n) = \LLhat\UUhat,
    \end{equation}
    where the entries on and below the diagonal, i.e., \(i\geq j\), satisfy
    \begin{equation} \label{eq:blocklu:getf2-lower}
    \begin{split}
        \Delta(\LL\UU)^{(\mathtt{getf2})}_{i,j}(n)
        ={}&-\sum_{k=1}^{j-1}
        \bigl(\Delta(\LL_{i,k}\UU_{k,j})^{(\mathtt{gemm})}(1)\bigr)^{(k)} \\
        &-\sum_{k=0}^{j-2}
        \bigl(\Delta\Acal^{(k)}_{i,j}\bigr)^{(\mathtt{gemm})}
        + \bigl(1-(1+\delta^{(\mathtt{div})}_{i,j})^{-1}\bigr)\LLhat_{i,j}\UUhat_{j,j},
    \end{split}
    \end{equation}
    while the entries above the diagonal, i.e., \(i< j\), satisfy \(\Delta(\LL_{ii}\UU_{ij})^{(\mathtt{trsv})}(1)=0\) and furthermore
    \begin{equation} \label{eq:blocklu:getf2-upper}
        \Delta(\LL\UU)^{(\mathtt{getf2})}_{i,j}(n)
        =-\sum_{k=1}^{i-1}\bigl(\Delta(\LL_{i,k}\UU_{k,j})^{(\mathtt{gemm})}(1)\bigr)^{(k)}
        -\sum_{k=0}^{i-2}\bigl(\Delta\Acal^{(k)}_{i,j}\bigr)^{(\mathtt{gemm})}.
    \end{equation}
\end{theorem}

\begin{proof}
    See Appendix~\ref{appendix:example:lu-proof-right}.
\end{proof}

\begin{theorem}[The left-looking block LU factorization] \label{thm:block-lu-left}
    Let \(\LL\in\mathbb{F}^{m\times n}\) and \(\UU\in\mathbb{F}^{n\times n}\)
    denote a lower trapezoidal matrix and an upper triangular matrix, respectively.
    Assume that \(\Acal = \LL\UU\) is computed by
    Algorithm~\ref{alg:block-LU-2}.
    Also, assume that Lines~\ref{line:lu:sub-lu}--\ref{line:update_trail} in
    Algorithm~\ref{alg:block-LU-2} satisfy, respectively,
    \eqref{eq:lem4:blocklu:trf2}, \eqref{eq:lem4:blocklu:trsm},
    \begin{equation} \label{eq:lem4:blocklu-2:row-gemm}
    \begin{split}
        \Acal_{i,i+1:p}-\LLhat_{i,1:i-1}\UUhat_{1:i-1,i+1:p}
        ={}&\Acalhat_{i,i+1:p}
        +\Delta(\LL_{i,1:i-1}\UU_{1:i-1,i+1:p})^{(\mathtt{gemm})}
        \bigl((i-1)n_b\bigr)\\
        &+\bigl(\Delta\Acal_{i,i+1:p}\bigr)^{(\mathtt{gemm})},
    \end{split}
    \end{equation}
    and
    \begin{equation} \label{eq:lem4:blocklu-2:gemm}
    \begin{split}
        \Acal_{i+1:q,i+1} - \LLhat_{i+1:q,1:i}\UUhat_{1:i,i+1}
        ={}& \Acalhat_{i+1:q,i+1}
        + \Delta\bigl(\LL_{i+1:q,1:i}\UU_{1:i,i+1}\bigr)^{(\mathtt{gemm})}(in_b) \\
        &+ \bigl(\Delta\Acal_{i+1:q,i+1}\bigr)^{(\mathtt{gemm})}.
    \end{split}
    \end{equation}
    Then there exists \(\Delta(\LL\UU)^{(\mathtt{getrf})}(n,n_b)\) such 
    that~\eqref{eq:blocklu:getrf} holds, where, for the \(i\)th block column with \(i=1,\ldots,p\),
    \begin{equation} \label{eq:blocklu-2:getrf-trackable}
    \begin{split}
        \Delta(\LL\UU)^{(\mathtt{getrf})}_{i:q,i}(n,n_b)
        ={}&-\Delta\bigl(\LL_{i:q,1:i-1}\UU_{1:i-1,i}\bigr)^{(\mathtt{gemm})}
        \bigl((i-1)n_b\bigr)\\
        &-\bigl(\Delta\Acal_{i:q,i}\bigr)^{(\mathtt{gemm})}
        +\Delta\bigl(\LL_{i:q,i}\UU_{i,i}\bigr)^{(\mathtt{getf2})}(n_b),
    \end{split}
    \end{equation}
    and, for the off-diagonal blocks in the \(i\)th block row with \(i=1,\ldots,p-1\),
    \begin{equation}
    \begin{split}
        \Delta(\LL\UU)^{(\mathtt{getrf})}_{i,i+1:p}(n,n_b)
        ={}&-\Delta(\LL_{i,1:i-1}\UU_{1:i-1,i+1:p})^{(\mathtt{gemm})}
        \bigl((i-1)n_b\bigr)\\
        &-\bigl(\Delta\Acal_{i,i+1:p}\bigr)^{(\mathtt{gemm})}
        +\Delta(\LL_{i,i}\UU_{i,i+1:p})^{(\mathtt{trsv})}(n_b).
        \label{eq:blocklu-2:getrf-trsv}
    \end{split}
    \end{equation}

    For the case \(n_b=1\), Algorithm~\ref{alg:block-LU-2} reduces to a
    non-block LU factorization.
    As in the conclusion for Algorithm~\ref{alg:block-LU}, 
    \texttt{XGETF2} and \texttt{XTRSV} in~\eqref{eq:lem4:blocklu:trf2} and~\eqref{eq:lem4:blocklu:trsm}, 
    respectively, reduce to divisions, 
    i.e.,~\eqref{eq:lem4:blocklu:trf2} reduces to~\eqref{eq:blocklu:div},
    and \(\Delta(\LL_{i,i}\UU_{i,i+1:p})^{(\mathtt{trsv})}(1)=0\) due to \(\LLhat_{i,i}=1\).
    Hence, there exists \(\Delta(\LL\UU)^{(\mathtt{getf2})}(n)\) such 
    that~\eqref{eq:blocklu:getf2} holds, where the entries on and below the 
    diagonal, i.e., \(i\geq j\), satisfy
    \begin{equation} \label{eq:blocklu-2:getf2-lower}
    \begin{split}
        \Delta(\LL\UU)^{(\mathtt{getf2})}_{i,j}(n)
        ={}&-\Delta\bigl(\LL_{i,1:j-1}\UU_{1:j-1,j}\bigr)^{(\mathtt{gemm})}(j-1) \\
        &-\bigl(\Delta\Acal_{i,j}\bigr)^{(\mathtt{gemm})}
        +\bigl(1-(1+\delta^{(\mathtt{div})}_{i,j})^{-1}\bigr)\LLhat_{i,j}\UUhat_{j,j},
    \end{split}
    \end{equation}
    while the entries above the diagonal, i.e., \(i<j\), satisfy \(\Delta(\LL_{ii}\UU_{ij})^{(\mathtt{trsv})}(1)=0\) and furthermore
    \begin{equation} \label{eq:blocklu-2:getf2-upper}
        \Delta(\LL\UU)^{(\mathtt{getf2})}_{i,j}(n)
        =-\Delta(\LL_{i,1:i-1}\UU_{1:i-1,j})^{(\mathtt{gemm})}(i-1)
        -\bigl(\Delta\Acal_{i,j}\bigr)^{(\mathtt{gemm})}.
    \end{equation}
\end{theorem}

\begin{proof}
    See Appendix~\ref{appendix:example:lu-proof-left}.
\end{proof}

The two variants organize the update errors differently.
In Algorithm~\ref{alg:block-LU}, the current block column or block row has
already passed through several trailing updates, so both the product errors
and the errors associated with adding each update to the current matrix appear
as sums over the preceding block steps.
In Algorithm~\ref{alg:block-LU-2}, the same preceding contributions are formed
in one longer matrix product and are added to the current matrix only once.
Consequently, a more accurate dot product can reduce the dimension dependent 
product error in the long matrix multiplication of the left-looking variant, 
whereas the repeated matrix update errors in the right-looking variant
remain.
This distinction is made quantitative by the following corollary.

The following corollary translates the trackable error representations in 
Theorems~\ref{thm:block-lu-right} and~\ref{thm:block-lu-left} into componentwise backward error bounds under the 
kernel assumptions summarized in Table~\ref{tab:notation-original}.


\begin{corollary}\label{cor:block-lu-componentwise}
Assume that the componentwise assumptions in Table~\ref{tab:notation-original}
hold with \(0\leq \epsgemmadd<1\) and \(0\leq \epsdiv<1\).
In particular, the triangular solves satisfy the componentwise backward error
bound
\[
    \abs{\Delta(U\hat X)^{(\mathtt{trsv})}(n)}\leq\epstrsv(n)\abs U\abs{\hat X}.
\]
Then the perturbations in Theorems~\ref{thm:block-lu-right} and~\ref{thm:block-lu-left} satisfy the following
componentwise bounds.
\begin{enumerate}
\item For Algorithm~\ref{alg:block-LU}, with \(p\geq2\),
\begin{equation}
\abs{\Delta(\LL\UU)^{(\mathtt{getrf})}(n,n_b)}
\leq \eta_{p-1}(n_b)\abs{\LLhat}\abs{\UUhat},
\label{eq:blocklu:getrf-componentwise}
\end{equation}
where
\begin{equation}
\eta_{p-1}(n_b):=\frac{1+\max\{\epsgemm(n_b),\epstrfvec(n_b),\epstrsv(n_b)\}}
{(1-\epsgemmadd)^{p-1}}-1.
\label{eq:blocklu:eta}
\end{equation}
For \(p=1\), only the panel factorization is required, and its bound is
\(\epstrfvec(n)\).

For the special case $n_b=1$, the scalar triangular solve
is exact under the unit-diagonal normalization, whereas the subdiagonal
multipliers are formed by divisions.
Then
\begin{equation} \label{eq:blocklu:getf2-componentwise}
\abs{\Delta(\LL\UU)^{(\mathtt{getf2})}(n)}
\leq\left(\frac{1+\max\left\{\epsgemm(1),\epsdiv/(1-\epsdiv)\right\}}
{(1-\epsgemmadd)^{n-1}}-1\right)\abs{\LLhat}\abs{\UUhat}.
\end{equation}

\item For Algorithm~\ref{alg:block-LU-2}, if \(p\geq2\) and
\(\epsgemm(n)\) is nondecreasing in \(n\), then
\begin{equation}
\begin{split}
&\abs{\Delta(\LL\UU)^{(\mathtt{getrf})}(n,n_b)} \\
&\qquad\leq\left(\frac{1+\max\{\epsgemm((p-1)n_b),\epstrfvec(n_b),\epstrsv(n_b)\}}
{1-\epsgemmadd}-1\right)\abs{\LLhat}\abs{\UUhat}.
\end{split}
\label{eq:blocklu-2:getrf-componentwise-monotone}
\end{equation}
For \(p=1\), only the panel factorization is required, and its bound is
\(\epstrfvec(n)\).

For the special case $n_b=1$ and \(n\geq2\),
\begin{equation}
\begin{split}
\abs{\Delta(\LL\UU)^{(\mathtt{getf2})}(n)}
\leq{}&\left(\frac{1+\max\left\{\epsgemm(n-1),\epsdiv/(1-\epsdiv)\right\}}{1-\epsgemmadd}-1\right)\abs{\LLhat}\abs{\UUhat}.
\end{split}
\label{eq:blocklu-2:getf2-componentwise}
\end{equation}
\end{enumerate}
\end{corollary}

\begin{proof}
    See Appendix~\ref{appendix:cor:block-lu-componentwise}.
\end{proof}

\begin{table}[!tb]
\centering
\caption{Componentwise assumptions for the local rounding errors in
Table~\ref{tab:notation}. The symbols \(\epsilon_{*}\) denote error constants
associated with the corresponding operations.}
\label{tab:notation-original}
\begin{spacing}{1.25}
\medskip
\resizebox{\textwidth}{!}{%
\begin{tabular}{cccc}
\hline
Operation & \begin{tabular}[c]{@{}c@{}}BLAS/LAPACK\\ Subroutine\end{tabular}
& Error quantity in Table~\ref{tab:notation} & Componentwise assumption\\ \hline
addition & -- & \(\delta^{(\mathtt{add})}\)
& \(\abs{\delta^{(\mathtt{add})}}\leq\epsadd\)\\
\hline
multiplication & -- & \(\delta^{(\mathtt{mult})}\)
& \(\abs{\delta^{(\mathtt{mult})}}\leq\epsmult\)\\
\hline
division & -- & \(\delta^{(\mathtt{div})}\)
& \(\abs{\delta^{(\mathtt{div})}}\leq\epsdiv\)\\
\hline
dot product & \texttt{XDOT}/\texttt{XDOTC} & \(\Delta x^{(\mathtt{dot})}(n)\)
& \(\abs{\Delta x^{(\mathtt{dot})}(n)}\leq\epsdot(n)\abs{x}\)\\
\hline
\begin{tabular}[c]{@{}c@{}}matrix--matrix\\ multiplication\end{tabular}
& \texttt{XGEMM}
& \begin{tabular}[c]{@{}c@{}}\(\Delta(AB)^{(\mathtt{gemm})}(n)\),\\ \((\Delta C)^{(\mathtt{gemm})}\)\end{tabular}
& \begin{tabular}[c]{@{}c@{}}\(\abs{\Delta(AB)^{(\mathtt{gemm})}(n)}\leq\epsgemm(n)\abs A\abs B\),\\
\(\abs{(\Delta C)^{(\mathtt{gemm})}}\leq
\epsgemmadd\abs{\alpha C}\)\end{tabular}\\
\hline
triangular solver & \texttt{XTRSM}/\texttt{XTRSV}
& \(\Delta(U\hat X)^{(\mathtt{trsv})}(n)\)
& \(\abs{\Delta(U\hat X)^{(\mathtt{trsv})}(n)}
\leq\epstrsv(n)\abs U\abs{\hat X}\)\\
\hline
\begin{tabular}[c]{@{}c@{}}non-block\\ LU factorization\end{tabular}
& \texttt{XGETF2}
& \(\Delta(LU)^{(\mathtt{getf2})}(n)\)
& \(\abs{\Delta(LU)^{(\mathtt{getf2})}(n)}\leq
\epstrfvec(n)\abs{\hat L}\abs{\hat U}\)\\ \hline
\end{tabular}}
\end{spacing}
\end{table}

\begin{remark}
For the right-looking block LU algorithm with a fixed block size
\(n_b\ll n\) and \(\epsgemmadd>0\),~\eqref{eq:blocklu:eta} shows that the 
matrix--matrix multiplication, panel factorization, and
triangular solve error constants depend only on \(n_b\), while the dependence
on the number of blocks, i.e., \(p\), enters explicitly through
\((1-\epsgemmadd)^{-(p-1)}\).
Thus, this bound identifies the repeatedly performed matrix subtractions as 
the source of error accumulation related to \(p\) (i.e., \(n\)) across block steps.
Computing these matrix subtractions in higher precision, or using an error-free transformation together with suitable compensated accumulation, can 
effectively reduce \(\epsgemmadd\) and hence mitigate this \(p\)-dependent
growth.
This observation is specific to the right-looking block LU algorithm.
In the left-looking variant, \(\epsgemm((p-1)n_b)\) also depends on the
increasing inner dimension of the matrix multiplication.
\end{remark}

\subsection{Error bounds for different computational kernels and arithmetic}
\label{subsec:diff-setting}
The trackable analysis above separates the global error structure from the error properties of the individual computational kernels. In this subsection, we demonstrate the modularity of the framework, instantiating the abstract error constants in Table~\ref{tab:notation-original} for different arithmetic and kernel implementations. This allows the corresponding LU factorization bounds to be obtained directly from the preceding results, while also showing how improved kernels can be incorporated into the analysis.

\subsubsection{Error bounds for real arithmetic using \(\gamman\)}
\label{remark:ourgamman-real}
We first specialize the componentwise assumptions in Table~\ref{tab:notation-original}
to real floating point arithmetic using our new notation \(\gamman=(1-\macheps)^{-n}-1\).
Applying Lemma~\ref{lem:gamma_property} to~\cite[(3.2)]{Higham2002} gives
\(\epsdot(n)=\gamman\).
For the scalar operations, \(\epsadd=\epsmult=\epsdiv=\macheps\).
For matrix--matrix multiplication, it should be noted that \(\epsgemm(n)\) 
and \(\epsgemmadd\) depend on the actual implementation.
Here, we first compute \(AB\), then form \(\alpha C\), and finally add the two matrices.%
\footnote{The reference BLAS \texttt{XGEMM} implementation,
 as documented on the LAPACK reference pages (https://www.netlib.org/lapack/explore-html/index.html), can use a different evaluation
 order.
 For example, in the no-transpose branch of \texttt{DGEMM}, \(C\) is first
 set to zero, left unchanged, or scaled according to \(\alpha\), and the products
 \(A_{ik}B_{kj}\) are then accumulated into \(C_{ij}\) one at a time.
 Under the usual model in which each multiplication and addition is rounded
 separately, this evaluation order leads to the bounds stated below.
 For the operation \(AB+\alpha C\) considered here, in real arithmetic, this gives \(\epsgemm(n)=\gamma_n\) for \(\alpha=0\)
 and \(\epsgemm(n)=\gamma_{n+1}\) otherwise.
 Valid bounds are \(\epsgemmadd(n)=0\), \(\gamma_n\), and
 \(\gamma_{n+1}\) for \(\alpha=0\), \(\alpha=1\), and
 \(\alpha\notin \{0,1\}\), respectively.
 Thus, under this implementation-specific model, the second constant is
 dimension dependent.
 The simpler evaluation order stated in the text is used throughout the
 analysis.}
Thus, a componentwise decomposition consistent with
Table~\ref{tab:notation-original} is obtained with
\[
    \epsgemm(n)=
    \begin{cases}
        \gamma_n, & \alpha=0,\\
        \gamma_{n+1}, & \alpha\neq0,
    \end{cases}
    \qquad\epsgemmadd=
    \begin{cases}
        0, & \alpha=0,\\
        \macheps, & \alpha=1,\\
        2\macheps+\macheps^2, & \alpha\notin \{0,1\}.
    \end{cases}
\]
The extra index in \(\epsgemm(n)\) for \(\alpha\neq0\) accounts for the
part of the final matrix addition that multiplies the computed \(AB\) term.
The constant \(\epsgemmadd\) accounts for the part associated with
\(\alpha C\).
For the LU updates below, the minus sign is absorbed into one factor of the
matrix product, so that \(\alpha=1\).
Hence, \(\epsgemm(n)=\gamma_{n+1}\) and \(\epsgemmadd=\macheps\) for these updates.

Substituting these constants into Lemma~\ref{lem:triangular-solver-xdot} gives
\(\epstrsv(n)=\max(\gamma_{n-1},\gamma_2)\) for Algorithm~\ref{alg:trsv-dot}, 
and into Lemma~\ref{lem:triangular-solver-xaxpy} gives \(\epstrsv(n)=\gamma_n\) for
Algorithm~\ref{alg:trsv-axpy}.
The bounds for the non-block LU factorization are obtained from the \(n_b=1\) 
specialization of Corollary~\ref{cor:block-lu-componentwise}, rather than 
directly from Theorems~\ref{thm:block-lu-right} and~\ref{thm:block-lu-left}.
For Algorithm~\ref{alg:block-LU}, the \(n_b=1\) specialization can be written directly in terms of our \(\gamma_n\) notation as
\[
    \epstrfvec(n)=\frac{1+\gamma_2}{(1-\macheps)^{n-1}}-1=\gamma_{n+1},
    \qquad n\geq2.
\]
For Algorithm~\ref{alg:block-LU-2}, Corollary~\ref{cor:block-lu-componentwise}
gives
\[
\epstrfvec(n)=\frac{1+\gamma_n}{1-\macheps}-1=\gamma_{n+1}.
\]
\begin{remark}
The extra index relative to a direct \(\gamma_n\) bound results from
applying Corollary~\ref{cor:block-lu-componentwise} to separate kernel-level
bounds, without retaining the correlation between the product and update
rounding errors.
As Tables~\ref{tab:notation} and~\ref{tab:notation-original} specify the
local error identities and their componentwise bounds, respectively, different
kernel assumptions may yield different rounding error bounds for the same
structural analysis.
Thus the two non-block variants have the same bound under the kernel-level
assumptions used here, although their trackable error decompositions remain
different.
\end{remark}

These results are summarized in Table~\ref{table:basicop-gamman-real}.

\begin{table}[!tb]
\centering
\renewcommand{\arraystretch}{1.5}
\caption{Floating point error bounds for basic real arithmetic operations using
our \(\gamma_n\) notation.  The kernel-level constants satisfy the componentwise
assumptions in Table~\ref{tab:notation-original}; the non-block LU bounds are
obtained by substituting these constants into Corollary~\ref{cor:block-lu-componentwise}
with \(n_b=1\).
For simplicity, we assume that \(AB+\alpha C\) is evaluated by first
computing the matrix product \(AB\), then forming \(\alpha C\), and finally
adding the two matrices.}
\label{table:basicop-gamman-real}
\begin{tabular}{c|c|c}
\hline
Operations & Algo & Floating-point error bounds \\
\hline
dot product & & \(\epsdot(n)=\gamma_n\) \\
\hline
addition & & \(\epsadd=\macheps\) \\
\hline
multiplication & & \(\epsmult=\macheps\) \\
\hline
division & & \(\epsdiv=\macheps\) \\
\hline
matrix--matrix multiplication & & \(
\begin{gathered}
\epsgemm(n)=
\begin{cases}
\gamma_n, & \alpha=0,\\
\gamma_{n+1}, & \alpha\neq0,
\end{cases}\\
\epsgemmadd=
\begin{cases}
0, & \alpha=0,\\
\macheps, & \alpha=1,\\
2\macheps+\macheps^2, & \alpha\notin \{0,1\}.
\end{cases}
\end{gathered}\) \\
\hline
\multirow{2}{*}{triangular solver}
& Algo~\ref{alg:trsv-dot} & \(\epstrsv(n)=\max\{\gamma_{n-1},\gamma_2\}\) \\
& Algo~\ref{alg:trsv-axpy} & \(\epstrsv(n)=\gamma_n\) \\
\hline
\multirow{2}{*}{non-block LU factorization}
& Algo~\ref{alg:block-LU} & \(\epstrfvec(n)=\gamma_{n+1}\) \\
& Algo~\ref{alg:block-LU-2} & \(\epstrfvec(n)=\gamma_{n+1}\) \\
\hline
\end{tabular}
\end{table}

\subsubsection{Error bounds for complex arithmetic using \(\gamman\)}
\label{remark:ourgamman}
We next consider complex arithmetic with \(\gamma_n=(1-\macheps)^{-n}-1\).
Assume that for complex numbers
\(x=a+\mi b\) and \(y=c+\mi d\), we compute
\begin{align}
    x\pm y &= a\pm c+\mi (b\pm d), \label{eq:complex:add-minus}\\
    xy &= ac-bd+\mi (ad+bc), \label{eq:complex:times}\\
    \frac{x}{y} &= \frac{ac+bd}{c^2+d^2}+\mi\frac{bc-ad}{c^2+d^2}.
    \label{eq:complex:div}
\end{align}

\begin{lemma}\label{lem:complex}
For \(x,y\in\mathbb C\), the basic arithmetic operations computed according
to~\eqref{eq:complex:add-minus}--\eqref{eq:complex:div} satisfy
\begin{align}
\fl(x\pm y)&=(x\pm y)(1+\delta), & \abs{\delta}&\leq\macheps,
\label{eq:lem-complex-er:add-minus}\\
\fl(xy)&=xy(1+\delta), & \abs{\delta}&\leq\sqrt{2}\,\gamma_2,
\label{eq:lem-complex-er:times}\\
\fl\bigl(x/y\bigr)&=(x/y)(1+\delta), & \abs{\delta}&\leq\sqrt{2}\,\gamma_5.
\label{eq:lem-complex-er:div}
\end{align}
\end{lemma}
\begin{proof}
See Appendix~\ref{appendix:example:lu-ourgamman}.
\end{proof}
Lemma~\ref{lem:complex} also holds with Higham's \(\gammanh\) notation.
The division bound depends on the implementation.
In particular, overflow-avoiding algorithms for complex and
quaternion division may introduce additional rounding errors
and thus require larger error bounds.

Using Lemma~\ref{lem:complex} together with the standard complex dot product
analysis gives the componentwise constants
\[
 \epsdot(n)=\sqrt2\gamma_{n+1},\quad
 \epsadd=\macheps,\quad
 \epsmult=\sqrt2\gamma_2,\quad
 \epsdiv=\sqrt2\gamma_5,
\]
and, for matrix--matrix multiplication,
\[
 \epsgemm(n)=
 \begin{cases}
 \sqrt2\gamma_{n+1}, & \alpha=0,\\
 \sqrt2\gamma_{n+2}, & \alpha\neq0,
 \end{cases}
 \qquad
 \epsgemmadd=
 \begin{cases}
 0, & \alpha=0,\\
 \macheps, & \alpha=1,\\
 \sqrt2\gamma_2+\macheps
 +\sqrt2\macheps\gamma_2, & \alpha\notin \{0,1\}.
 \end{cases}
\]
For the LU updates, \(\alpha=1\), and therefore \(\epsgemm(n)=\sqrt2\gamma_{n+2}\)
and \(\epsgemmadd=\macheps\).

For the two triangular solvers, Lemmas~\ref{lem:triangular-solver-xdot}
and~\ref{lem:triangular-solver-xaxpy} give the following bounds for \(n\geq2\).
For Algorithm~\ref{alg:trsv-dot}, we take
\begin{equation*}
\begin{split}
\epstrsv(n)&=\max\left\{\sqrt2\gamma_n,\frac{1}{(1-\sqrt2\gamma_5)(1-\macheps)}-1\right\} \\
&=\max\left\{\sqrt2\gamma_n,\frac{\gamma_1+\sqrt2\gamma_5}{1-\sqrt2\gamma_5}\right\}\\
&\leq\frac{\sqrt2}{1-\sqrt2\gamma_5}\gamma_{\max\{n,6\}}.
\end{split}
\end{equation*}
For Algorithm~\ref{alg:trsv-axpy}, we have
\begin{equation*}
\begin{split}
\epstrsv(n)
&=\frac{1}{(1-\sqrt2\gamma_5)(1-\macheps)^{n-1}}-1
=\frac{\gamma_{n-1}+\sqrt2\gamma_5}{1-\sqrt2\gamma_5}
\leq\frac{\sqrt2}{1-\sqrt2\gamma_5}\gamma_{n+4}.
\end{split}
\end{equation*}
Here the last inequalities use~\eqref{eq:lem-property:4} shown in Lemma~\ref{lem:gamma_property}, which yields
\(\gamma_{n-1}+\sqrt2\gamma_5\leq\sqrt2\gamma_{n+4}\) and 
\(\gamma_1+\sqrt2\gamma_5\leq\sqrt2\gamma_6\).

For the non-block LU factorization, the \(n_b=1\) case of
Corollary~\ref{cor:block-lu-componentwise} gives
\[
 \epstrfvec(n)
 =\frac{1}{(1-\sqrt2\gamma_5)(1-\macheps)^{n-1}}-1
 =\frac{\gamma_{n-1}+\sqrt2\gamma_5}{1-\sqrt2\gamma_5}
 \leq\frac{\sqrt2}{1-\sqrt2\gamma_5}\gamma_{n+4},
 \qquad n\geq2,
\]
for Algorithm~\ref{alg:block-LU}, which requires \(\sqrt2\gamma_5<1\).
For Algorithm~\ref{alg:block-LU-2}, substituting
\(\epsgemm(n-1)=\sqrt2\gamma_{n+1}\),
\(\epsgemmadd=\macheps\), and
\(\epsdiv=\sqrt2\gamma_5\) into
\eqref{eq:blocklu-2:getf2-componentwise}, we have
\begin{equation*}
\begin{split}
\epstrfvec(n)
&=\frac{1+\max\left\{\sqrt2\gamma_{n+1},
\dfrac{\sqrt2\gamma_5}{1-\sqrt2\gamma_5}\right\}}{1-\macheps}-1\\
&=\max\left\{\frac{1+\sqrt2\gamma_{n+1}}{1-\macheps}-1,
\frac{\gamma_1+\sqrt2\gamma_5}{1-\sqrt2\gamma_5}\right\}\\
&\leq\frac{\sqrt2}{1-\sqrt2\gamma_5}\gamma_{\max\{n+2,6\}},
\end{split}
\end{equation*}
for \(n\geq2\).
It should be noted that all four bounds above require only \(\sqrt2\gamma_5<1\).
If, in addition, \(\sqrt2\gamma_5\leq1/2\), which is ensured, for example,
by the dimension-independent condition \(\macheps\leq1/20\), the factor
\(\sqrt2/(1-\sqrt2\gamma_5)\) is at most \(2\sqrt2\).
We summarize these simplified bounds in Table~\ref{table:basicop-gamman-complex}.

\begin{remark}
Although our notation \(\gamma_n=(1-\macheps)^{-n}-1\)
does not eliminate the need for small-roundoff assumptions,
the complex TRSV and non-block LU bounds above require only
\(\sqrt2\gamma_5<1\), or the stronger condition
\(\sqrt2\gamma_5\leq1/2\) when the constants are further simplified.
Both conditions depend only on \(\macheps\), not on the matrix dimension \(n\).
In contrast, Higham's notation
\(\gammanh=n\macheps/(1-n\macheps)\) requires \(n\macheps<1\).
Thus, our conditions do not impose dimension-dependent restrictions 
on \(n\) and are considerably less restrictive for large-scale
problems.
\end{remark}

\begin{remark}
A common simplification in rounding error analysis is to
use a single worst-case scalar operation bound for addition, multiplication,
and division.
For complex arithmetic, this amounts to using the division
bound \(\sqrt2\gamma_5\) for all three operations, although 
\(\epsadd=\macheps\), \(\epsmult=\sqrt2\gamma_2\), and
\(\epsdiv=\sqrt2\gamma_5\).
Since the division bound is substantially larger, applying it uniformly
to repeated additions and multiplications can lead to an overestimate of the 
total rounding error.
Our trackable analysis retains these operation-specific bounds separately
throughout the error propagation, which enables tighter triangular solve and
LU bounds when this distinction is significant.
\end{remark}

\begin{table}[!tb]
\centering
\renewcommand{\arraystretch}{1.5}
\caption{Floating point error bounds for basic complex arithmetic operations
using our \(\gamma_n\) notation.  The kernel-level constants satisfy the
componentwise assumptions in Table~\ref{tab:notation-original}; the non-block
LU bounds follow from Corollary~\ref{cor:block-lu-componentwise} with \(n_b=1\).
For simplicity, we assume that \(AB+\alpha C\) is evaluated by first
computing the matrix product \(AB\), then forming \(\alpha C\), and finally
adding the two matrices.
For triangular solver and non-block LU factorization, we additionally assume
\(\sqrt2\gamma_5\leq1/2\), which holds in particular for
\(\macheps\leq1/20\).}
\label{table:basicop-gamman-complex}
\begin{tabular}{c|c|c}
\hline
Operations & Algo & Floating-point error bounds \\
\hline
dot product & & \(\epsdot(n)=\sqrt2\gamma_{n+1}\) \\
\hline
addition & & \(\epsadd=\macheps\) \\
\hline
multiplication & & \(\epsmult=\sqrt2\gamma_2\) \\
\hline
division & & \(\epsdiv=\sqrt2\gamma_5\) \\
\hline
\shortstack[c]{matrix--matrix\\multiplication} &
& \(\begin{gathered}
\epsgemm(n)=
\begin{cases}
\sqrt2\gamma_{n+1}, & \alpha=0,\\
\sqrt2\gamma_{n+2}, & \alpha\neq0,
\end{cases}\\
\epsgemmadd=
\begin{cases}
0, & \alpha=0,\\
\macheps, & \alpha=1,\\
\sqrt2\gamma_2+\macheps+\sqrt2\macheps\gamma_2,
& \alpha\notin \{0,1\}
\end{cases}
\end{gathered}\) \\
\hline
\multirow{2}{*}{triangular solver}
& Algo~\ref{alg:trsv-dot} & \(\epstrsv(n)=2\sqrt2\gamma_{\max\{n,6\}}\) \\
& Algo~\ref{alg:trsv-axpy} & \(\epstrsv(n)=2\sqrt2\gamma_{n+4}\) \\
\hline
\multirow{2}{*}{\shortstack[c]{non-block LU\\ factorization}}
& Algo~\ref{alg:block-LU} & \(\epstrfvec(n)=2\sqrt2\gamma_{n+4}\) \\
& Algo~\ref{alg:block-LU-2} & \(\epstrfvec(n)=2\sqrt2\gamma_{\max\{n+2,6\}}\) \\
\hline
\end{tabular}
\end{table}

\subsubsection{Error bounds for quaternion arithmetic using \(\gamman\)}
\label{remark:quaternion}
Finally, we consider quaternion arithmetic using \(\gamman\).
For
\(x=x_1+\mi x_2+\mj x_3+\mk x_4\) and \(y=y_1+\mi y_2+\mj y_3+\mk y_4\), assume
\begin{align}
 x\pm y&=x_1\pm y_1+\mi(x_2\pm y_2)+\mj(x_3\pm y_3)+\mk(x_4\pm y_4),
 \label{eq:quaternion:add-minus}\\
 xy&=(x_1y_1-x_2y_2-x_3y_3-x_4y_4)
 +\mi(x_1y_2+x_2y_1+x_3y_4-x_4y_3)\notag\\
 &\quad+\mj(x_1y_3-x_2y_4+x_3y_1+x_4y_2)
 +\mk(x_1y_4+x_2y_3-x_3y_2+x_4y_1),
 \label{eq:quaternion:times}\\
 xy^{-1}&=x\Bar y\abs y^{-2}.
 \label{eq:quaternion:div}
\end{align}

\begin{lemma}\label{lem:quaternion}
For quaternion numbers \(x\) and \(y\), the operations above satisfy
\begin{align}
 \fl(x\pm y)&=(x\pm y)(1+\delta_r)=(1+\delta_\ell)(x\pm y),
 & \abs{\delta_r},\abs{\delta_\ell}&\leq\macheps,
 \label{eq:lem-quaternion-er:add-minus}\\
 \fl(xy)&=xy(1+\delta_r)=(1+\delta_\ell)xy=x(1+\delta_m)y,
 & \abs{\delta_r},\abs{\delta_\ell}, \abs{\delta_m}&\leq2\gamma_4,
 \label{eq:lem-quaternion-er:times}\\
 \fl\bigl(xy^{-1}\bigr)&=(xy^{-1})(1+\delta_r)=(1+\delta_\ell)(xy^{-1})=x(1+\delta_m)y^{-1},
 & \abs{\delta_r},\abs{\delta_\ell}, \abs{\delta_m}&\leq\gamma_{13}.
 \label{eq:lem-quaternion-er:div}
\end{align}
Here \(\delta_r\), \(\delta_\ell\), and \(\delta_m\) need not be the same quaternion.
All representations are stated explicitly because quaternion multiplication is noncommutative.
\end{lemma}
\begin{proof}
See Appendix~\ref{appendix:example:lu-quaternion}.
\end{proof}

Since the bounds in Lemma~\ref{lem:quaternion} hold with the relative error
factor on either the right, the left, or the middle, the same constants can be used in the
kernel models below regardless of which side is required by the corresponding
quaternion operation.
This point is essential because quaternion multiplication is not commutative.

The resulting componentwise kernel constants are
\[
 \epsdot(n)=2\gamma_{n+3},\quad
 \epsadd=\macheps,\quad
 \epsmult=2\gamma_4,\quad
 \epsdiv=\gamma_{13},
\]
and
\[
 \epsgemm(n)=
 \begin{cases}
 2\gamma_{n+3}, & \alpha=0,\\
 2\gamma_{n+4}, & \alpha\neq0,
 \end{cases}
 \qquad
 \epsgemmadd=
 \begin{cases}
 0, & \alpha=0,\\
 \macheps, & \alpha=1,\\
 2\gamma_4+\macheps+2\macheps\gamma_4,
 & \alpha\notin \{0,1\}.
 \end{cases}
\]

The triangular solve bounds follow from Lemmas~\ref{lem:triangular-solver-xdot} and~\ref{lem:triangular-solver-xaxpy}.
For Algorithm~\ref{alg:trsv-dot}, we have
\begin{equation*}
\begin{split}
\epstrsv(n)
&=\max\left\{2\gamma_{n+2},
\frac{1}{(1-\gamma_{13})(1-\macheps)}-1\right\}
\leq\max\left\{2\gamma_{n+2},
\frac{\gamma_{14}}{1-\gamma_{13}}\right\}.
\end{split}
\end{equation*}
For Algorithm~\ref{alg:trsv-axpy}, we obtain
\begin{equation*}
\begin{split}
\epstrsv(n)
&=\max\left\{
\frac{1+2\gamma_4}{(1-\macheps)^{n-2}}-1,
\frac{1}{(1-\gamma_{13})(1-\macheps)^{n-1}}-1
\right\}\\
&=\frac{1}{(1-\gamma_{13})(1-\macheps)^{n-1}}-1
\leq\frac{1}{1-\gamma_{13}}\gamma_{n+12}.
\end{split}
\end{equation*}
For the non-block LU factorization, substitution into the \(n_b=1\) case of
Corollary~\ref{cor:block-lu-componentwise} proves
\begin{equation*}
\begin{split}
\epstrfvec(n)
&=\frac{1}{(1-\gamma_{13})(1-\macheps)^{n-1}}-1
=\frac{\gamma_{n-1}+\gamma_{13}}{1-\gamma_{13}}
\leq\frac{1}{1-\gamma_{13}}\gamma_{n+12},
\end{split}
\label{eq:quaternion-getf2-right}
\end{equation*}
for Algorithm~\ref{alg:block-LU}, and
\begin{equation*}
\begin{split}
\epstrfvec(n)
&=\max\left\{2\gamma_{n+3},
\frac{\gamma_{13}}{1-\gamma_{13}}\right\}
\frac{1}{1-\macheps}+\frac{\macheps}{1-\macheps}
\leq\max\left\{2\gamma_{n+4},
\frac{\gamma_{14}}{1-\gamma_{13}}\right\}.
\end{split}
\label{eq:quaternion-getf2-left}
\end{equation*}
for Algorithm~\ref{alg:block-LU-2}.
All four bounds require \(\gamma_{13}<1\).
Furthermore, if \(\gamma_{13}\leq 1/2\), which is ensured by \(\macheps\leq 1/33\), then the factor \(1/(1-\gamma_{13})\) is at most \(2\).
We summarize these simplified bounds in Table~\ref{table:basicop-gamman-quaternion}.

\begin{table}[!tb]
\centering
\renewcommand{\arraystretch}{1.5}
\caption{Floating point error bounds for basic quaternion arithmetic operations
using our \(\gamman\) notation. 
The kernel-level constants satisfy the componentwise assumptions in 
Table~\ref{tab:notation-original}; the non-block LU bounds follow from 
Corollary~\ref{cor:block-lu-componentwise} with \(n_b=1\).
For simplicity, we assume that \(AB+\alpha C\) is evaluated by first
computing the matrix product \(AB\), then forming \(\alpha C\), and finally
adding the two matrices.
For triangular solver and non-block LU factorization, we additionally assume
\(\gamma_{13}\leq1/2\), which holds in particular for
\(\macheps\leq1/33\).}
\label{table:basicop-gamman-quaternion}
\begin{tabular}{c|c|c}
\hline
Operations & Algo & Floating-point error bounds \\
\hline
dot product & & \(\epsdot(n)=2\gamma_{n+3}\) \\
\hline
addition & & \(\epsadd=\macheps\) \\
\hline
multiplication & & \(\epsmult=2\gamma_4\) \\
\hline
division & & \(\epsdiv=\gamma_{13}\) \\
\hline
\shortstack[c]{matrix--matrix\\ multiplication} & & \(\begin{gathered}
\epsgemm(n)=
\begin{cases}
2\gamma_{n+3}, & \alpha=0,\\
2\gamma_{n+4}, & \alpha\neq0,
\end{cases}\\
\epsgemmadd=
\begin{cases}
0, & \alpha=0,\\
\macheps, & \alpha=1,\\
2\gamma_4+\macheps+2\macheps\gamma_4,
& \alpha\notin \{0,1\}
\end{cases}
\end{gathered}\) \\
\hline
\multirow{2}{*}{triangular solver}
& Algo~\ref{alg:trsv-dot} & \(\epstrsv(n)=2\gamma_{\max\{n+2,14\}}\) \\
& Algo~\ref{alg:trsv-axpy} & \(\epstrsv(n)=2\gamma_{n+12}\) \\
\hline
\multirow{2}{*}{\shortstack[c]{non-block LU\\ factorization}}
& Algo~\ref{alg:block-LU} & \(\epstrfvec(n)=2\gamma_{n+12}\) \\
& Algo~\ref{alg:block-LU-2} & \(\epstrfvec(n)=2\gamma_{\max\{n+4,14\}}\) \\
\hline
\end{tabular}
\end{table}

\subsubsection{Improved bounds based on Kahan compensated summation}\label{remark:Kahan}
    Kahan compensated summation~\cite{Kahan1965} is a classical technique for 
    improving the accuracy of floating point summation.
    It introduces a compensation term to recover low-order bits lost during rounding,
    thereby reducing the leading error term from \(\bigO(n\macheps)\) to \(\bigO(\macheps)\).
    Based on compensated summation, a more accurate floating point dot 
    product can be constructed by first computing the floating-point 
    products and then applying Kahan compensated summation to the resulting sequence.
    Compared with the standard recursive dot product, this approach 
    suppresses the linear growth of the leading rounding error with respect 
    to the vector length.
    
    For Kahan compensated summation, \cite[(4.9)]{Higham2002} gives the
    error bound \(2\macheps+\bigO(n\macheps^2)\) for summing a sequence
    of floating-point numbers.
    For a dot product, the products are first formed in floating-point
    arithmetic, introducing an additional relative error bounded by
    \(\macheps\).
    Combining the multiplication errors with the compensated summation
    errors therefore gives \(\epsdot(n) = 3\macheps+\bigO(n\macheps^2)\).
    If matrix--matrix multiplication is implemented as a straightforward 
    collection of such dot products followed by the usual floating-point 
    matrix update, the two error sources in Table~\ref{tab:notation-original}
    satisfy
    \[
    \epsgemm(n)=
    \begin{cases}
    3\macheps+\bigO(n\macheps^2), & \alpha=0,\\
    4\macheps+\bigO(n\macheps^2), & \alpha\neq0,
    \end{cases}
    \qquad\epsgemmadd=
    \begin{cases}
    0, & \alpha=0,\\
    \macheps, & \alpha=1,\\
    2\macheps+\macheps^2, & \alpha\notin \{0,1\}.
    \end{cases}
    \]
    For the LU updates, i.e., \(\alpha=1\), then
    \(\epsgemm(n)=4\macheps+\bigO(n\macheps^2)\) and
    \(\epsgemmadd=\macheps\).
    Consequently, when Algorithm~\ref{alg:trsv-dot} is implemented using 
    Kahan compensated dot products, its error bound satisfies \(\epstrsv(n) = 3\macheps+\bigO(n\macheps^2)\).
    Detailed derivations of these bounds are provided in Appendix~\ref{appendix:example:lu-Kahan}.

    It is worth noting that the effectiveness of compensated summation 
    depends on the source of the dimension-dependent factor in the error 
    bound.
    For Algorithm~\ref{alg:trsv-axpy}, which is based on \texttt{XAXPY} 
    operations, the factor \(n\) in the leading error term results from the 
    accumulation of errors over successive \texttt{XAXPY} updates rather 
    than from a long summation.
    Hence, applying Kahan compensated summation does not remove this 
    dependence on \(n\).
    A similar phenomenon occurs for Algorithm~\ref{alg:block-LU} for LU 
    factorization.
    Although compensated summation can improve the accuracy of each 
    individual matrix--matrix multiplication, Algorithm~\ref{alg:block-LU} 
    performs multiple successive trailing-matrix updates, and the factor \(p-1\)
    arising from the accumulation of these update errors remains.
    Thus, its leading error bound still depends on the number of block updates.

    In contrast, Algorithm~\ref{alg:block-LU-2} accumulates the 
    contributions from the preceding blocks within matrix--matrix 
    multiplications whose inner dimensions increase with the iteration index.
    Applying Kahan compensated summation to the dot products underlying 
    these matrix--matrix multiplications reduces the corresponding leading 
    error from \(O(n\macheps)\) to \(O(\macheps)\).
    Therefore, when Kahan compensated summation is used for the dot products 
    and matrix--matrix multiplications in Algorithm~\ref{alg:block-LU-2}, 
    together with Algorithm~\ref{alg:trsv-dot} for the triangular solves, 
    the dependence on \(n\) can be removed from the leading term of the 
    resulting error bound.
    The corresponding floating-point error bounds are summarized in Table~\ref{table:lu}.

    We emphasize that Kahan compensated summation is used here only as an illustrative example.
    The same trackable analysis can be readily combined with other accurate
    summation techniques.
    For instance, doubly compensated summation employs an additional level of compensation
    to further reduce the higher-order rounding errors remaining in Kahan compensated summation.
    More generally, once an error bound for a particular summation or dot product
    scheme is available, it can be incorporated directly into the assumptions of
    our framework, allowing its effect on different algorithmic variants to be analyzed
    without repeating the rounding error analysis of the entire algorithm.
    Similar considerations apply to other accurate summation algorithms, such as iFastSum~\cite{ZH2009}.

\begin{table}[!tb]
\centering
\renewcommand{\arraystretch}{1.5}
\caption{Comparison between componentwise error bounds for the triangular
solver and LU factorization with and without Kahan compensated summation.
The constants are substituted into Corollary~\ref{cor:block-lu-componentwise}.
The symbol \(\star\) indicates that the dot products underlying XDOT and XGEMM
are evaluated using Kahan compensated summation.  For standard real arithmetic,
the kernel constants are those in Table~\ref{table:basicop-gamman-real}.
Boldface indicates the algorithm with the smallest error bound for each operation.}
\label{table:lu}
\resizebox{\textwidth}{!}{%
\begin{tabular}{c|c|c}
\hline
Operations & Algorithms & Floating-point error bounds \\
\hline
\multirow{3}{*}{triangular solver}
& Algo~\ref{alg:trsv-dot} & \(\epstrsv(n)=\max\{\gamma_{n-1},\gamma_2\}\) \\
& Algo~\ref{alg:trsv-axpy} & \(\epstrsv(n)=\gamma_n\) \\
& \textbf{Algo~\ref{alg:trsv-dot}\(^\star\)} & \(\epstrsv(n)=3\macheps+\bigO(n\macheps^2)\) \\
\hline
\multirow{3}{*}{\begin{tabular}[c]{@{}c@{}}non-block\\ LU factorization\end{tabular}}
& Algo~\ref{alg:block-LU} (\(n_b=1\)) & \(\epstrfvec(n)=\gamma_{n+1}\) \\
& Algo~\ref{alg:block-LU-2} (\(n_b=1\)) & \(\epstrfvec(n)=\gamma_{n+1}\) \\
& \textbf{Algo~\ref{alg:block-LU-2}\(^\star\) (\(n_b=1\))} & \(\epstrfvec(n)=5\macheps+\bigO(n\macheps^2)\) \\
\hline
\multirow{4}{*}{\begin{tabular}[c]{@{}c@{}}block\\ LU factorization\end{tabular}}
& Algo~\ref{alg:block-LU}+Algo~\ref{alg:trsv-dot}/~\ref{alg:trsv-axpy}
& \(\epstrf(n,n_b,p)=\gamma_{n_b+p}\) \\
& Algo~\ref{alg:block-LU-2}+Algo~\ref{alg:trsv-dot}/~\ref{alg:trsv-axpy}
& \(\epstrf(n,n_b,p)=\gamma_{n-n_b+2}\) \\
& Algo~\ref{alg:block-LU}\(^\star\)+Algo~\ref{alg:trsv-dot}\(^\star\)
& \(\epstrf(n,n_b,p)=(p+4)\macheps
+\bigO((n_b+p^2)\macheps^2)\) \\
& \textbf{Algo~\ref{alg:block-LU-2}\(^\star\)+Algo~\ref{alg:trsv-dot}\(^\star\)}
& \(\epstrf(n,n_b,p)=6\macheps+\bigO(n\macheps^2)\) \\
\hline
\end{tabular}}
\end{table}

\subsubsection{Improved bounds based on the Ozaki matrix multiplication}
\label{remark:ozaki}
Ozaki, Ogita, Oishi, and Rump~\cite{Ozaki2012} proposed an accurate
matrix--matrix multiplication based on error-free splitting.
This approach improves the accuracy of standard matrix multiplication while 
performing the dominant computational work using highly optimized 
\texttt{XGEMM} routines in the working precision.

Let \(C\) denote the result obtained with \(k\) splitting levels and
faithfully rounded final summation.
Under the row- and column-wise magnitude condition in~\cite[Corollary~1]{Ozaki2012},
\begin{equation} \label{eq:ozaki-gemm-bound}
 \abs{C-AB}\leq \varepsilon_{\mathrm{Oz}}(n)\abs A\abs B,
 \qquad\varepsilon_{\mathrm{Oz}}(n)
 =2^{(\beta-h)(k-1)+2\omega+1}\gammanh+2\macheps,
\end{equation}
where
\[
h=-\log_2\macheps,\qquad \beta=\left\lceil\frac{\log_2n+h}{2}\right\rceil.
\]
Here \(\omega>0\) characterizes the magnitude spread of the
entries of \(A\) and \(B\), with the relevant magnitude ratios
bounded between \(2^{-\omega}\) and \(2^\omega\);
see~\cite[(42)]{Ozaki2012}.
It is straightforward to verify that \(\beta-h<0\) if \(4n\macheps\leq 1\).
Therefore, one may increase \(k\) to decrease the \(\bigO(n\macheps)\) term 
toward \(\bigO(\macheps)\), although a larger \(k\) also raises the 
computational cost.
Consequently, Ozaki multiplication can yield a substantially smaller error than 
the standard \(n\macheps/\bigl(1-n\macheps\bigr)\) level matrix multiplication.

We illustrate this improvement for block LU by using Ozaki multiplication
only for the trailing-update \texttt{XGEMM}.
For simplicity, the non-block panel factorization and triangular solve are bounded by \(\gammanh\) by~\cite{Higham2002}.
If the update with the trailing matrix is performed by ordinary floating point 
addition, then
\[
\epsgemm(n)=(1+\macheps)\varepsilon_{\mathrm{Oz}}(n)+\macheps,
\qquad
\epsgemmadd=\macheps.
\]

For Algorithm~\ref{alg:block-LU}, Corollary~\ref{cor:block-lu-componentwise} then gives
\begin{equation} \label{eq:ozaki-getrf-alg1}
 \abs{\Delta(\LL\UU)^{(\mathtt{getrf})}(n,n_b)}
 \leq\left(\frac{1+\max\left\{
 (1+\macheps)\varepsilon_{\mathrm{Oz}}(n_b)+\macheps,
 \gamma_{n_b}^{\mathrm h}\right\}}{(1-\macheps)^{p-1}}-1\right)\abs{\LLhat}\abs{\UUhat}.
\end{equation}
For fixed \(n_b\), \(\varepsilon_{\mathrm{Oz}}(n_b)\) is independent of \(p\).
Consequently, even if Ozaki multiplication reduces its matrix-product
contribution to \(\bigO(\macheps)\), the repeated updates still produce an
\(\bigO(p\macheps)\) leading term through the factors associated with the
successive additions to the trailing matrix.

For Algorithm~\ref{alg:block-LU-2}, Corollary~\ref{cor:block-lu-componentwise} 
yields
\begin{equation} \label{eq:ozaki-getrf-alg2}
 \abs{\Delta(\LL\UU)^{(\mathtt{getrf})}(n,n_b)}
 \leq\left(\frac{1+\max\left\{
 (1+\macheps)\varepsilon_{\mathrm{Oz}}\bigl((p-1)n_b\bigr)
 +\macheps,
 \gamma_{n_b}^{\mathrm h}\right\}}{1-\macheps}-1\right)
 \abs{\LLhat}\abs{\UUhat},
\end{equation}
and this makes it possible to improve the bound to \(\bigO\bigl(n_b\macheps\bigr)\).
Note that the choice of \(n_b\) is typically independent of \(n\).

\subsubsection{Improved bounds for IEEE arithmetic}
\label{remark:ufp}
We now return to standard real floating-point arithmetic, assuming rounding to 
nearest and the absence of underflow and overflow.
The preceding results are obtained by substituting kernel-level error bounds 
into Corollary~\ref{cor:block-lu-componentwise}.
The improved analysis of Rump and Jeannerod in~\cite{RJ2014} is of a different 
nature and cannot, in general, be recovered from that corollary: decomposing 
the computation into separate \texttt{XGEMM}, triangular solve, and panel 
factorization bounds introduces products of local error factors and therefore 
destroys the purely linear dependence on \(\macheps\).

Instead, we return to the trackable identities in Theorems~\ref{thm:block-lu-right} and~\ref{thm:block-lu-left},
and bound the complete arithmetic expressions that produce the entries of the 
computed factors.
The improved analysis of IEEE arithmetic LU is based on~\cite[Lemma~4.1]{RJ2014}.
If
\[
 y=\frac{c-\sum_{r=1}^{k-1}a_rb_r}{b_k}
\]
is evaluated in floating point arithmetic, then, in the absence of underflow
and overflow and independently of the order of evaluation, the computed
\(\widehat y\) satisfies
\[
 b_k\widehat y(1+\theta_k^{(0)})=c-\sum_{r=1}^{k-1}a_rb_r(1+\theta_k^{(r)}),
 \qquad\abs{\theta_k^{(r)}}\le k\macheps.
\]
When \(b_k=1\), the bound improves to \((k-1)\macheps\).
It is important that the terms involved on the right-hand side of~\eqref{eq:blocklu:getrf-trackable}
and~\eqref{eq:blocklu-2:getrf-trackable} cannot be bounded separately,
since~\cite[Lemma~4.1]{RJ2014} cannot provide independent
GEMM,
matrix addition, and \texttt{XGETF2} bounds having the same purely linear constant.

This result can be applied directly to the entrywise expressions underlying
LU factorization.
Theorems~\ref{thm:block-lu-right} and~\ref{thm:block-lu-left} have already given the global trackable backward relation
\[
    \Acal+\Delta(\LL\UU)^{(\mathtt{getrf})}(n,n_b)
    =\LLhat\UUhat.
\]
Rather than bounding the individual terms in
\(\Delta(\LL\UU)^{(\mathtt{getrf})}\) separately as in
Corollary~\ref{cor:block-lu-componentwise}, we regard each computed entry as
an evaluation, in the order prescribed by the particular LU variant, of the
same Gaussian elimination expression.
For the \(k\)th elimination step, the entries in the \(k\)th row of \(\widehat U\)
involve no division and therefore satisfy the \((k-1)\macheps\) bound of
\cite[Lemma~4.1]{RJ2014}; the entries below the diagonal in the \(k\)th
column of \(\widehat L\) involve one division and satisfy the
\(k\macheps\) bound.
Consequently the perturbation in Theorems~\ref{thm:block-lu-right} 
and~\ref{thm:block-lu-left} obeys exactly the Rump--Jeannerod bound shown 
in~\cite[Theorem~4.2]{RJ2014}:
\[
    \abs{\Delta(\LL\UU)^{(\mathtt{getrf})}(n,n_b)}
    \le n\macheps\abs{\LLhat}\abs{\UUhat},
\]
for \(\Acal\in\mathbb F^{m\times n}\), \(m\ge n\).  If
\(m=n\), the sharper row-dependent result is
\[
    \abs{\Delta(\LL\UU)^{(\mathtt{getrf})}(n,n_b)}
    \le
    \operatorname{diag}\bigl(0,\macheps,\ldots,(n-1)\macheps\bigr)
    \abs{\LLhat}\abs{\UUhat}
    \le (n-1)\macheps\abs{\LLhat}\abs{\UUhat}.
\]
Thus Theorems~\ref{thm:block-lu-right} and~\ref{thm:block-lu-left} supply the trackable global perturbation,
while the whole-expression result of Rump and Jeannerod supplies a sharper
bound for that perturbation than is obtained by applying the kernel-level
assumptions of Table~\ref{tab:notation-original} separately.

Crucially, \cite[Lemma~4.1]{RJ2014} is independent of the order of evaluation.
Hence the same final bound applies to Algorithms~\ref{alg:block-LU} and
\ref{alg:block-LU-2}, including their blocked and non-blocked specializations:
the two algorithms organize the arithmetic differently, but this only changes
the evaluation order of the underlying Gaussian elimination expressions.
This also explains why this result should be derived from Theorems~\ref{thm:block-lu-right} and~\ref{thm:block-lu-left}
rather than from Corollary~\ref{cor:block-lu-componentwise}: the latter is
designed to preserve kernel-level information and distinguish algorithmic
variants, whereas the Rump--Jeannerod whole-expression model deliberately
retains enough arithmetic structure to obtain a purely linear bound in
\(\macheps\).

\section{Concluding Remarks}
\label{sec:conclusions}
In this work, we develop a framework for trackable and modular rounding error analysis.
A key component is the notation
\(\gamman=(1-\macheps)^{-n}-1\).
It removes the restriction \(n\macheps<1\) imposed by the standard
\(\gammanh\)-notation.
This does not mean that all error bounds are free of roundoff assumptions.
For instance, the complex and quaternion bounds in
Sections~\ref{remark:ourgamman} and~\ref{remark:quaternion} require
\(\sqrt2\gamma_5<1\) and \(\gamma_{13}<1\), respectively, and stronger
requirements \(\sqrt2\gamma_5\leq 1/2\) and \(\gamma_{13}\leq 1/2\) are adopted when the constants are further simplified.
Importantly, these conditions depend only on \(\macheps\), not on the problem dimension
\(n\), and therefore can be substantially less restrictive than \(n\macheps<1\)
in large-scale applications.
Meanwhile, the notation preserves useful algebraic properties for accumulated
rounding errors and naturally incorporates errors arising from division.
As a result, it is particularly effective for large-scale settings or computations with low precision.

A main feature of the framework is that it retains the source and propagation
of individual rounding errors throughout the analysis, i.e., the rounding errors are trackable.
This information is generally lost when local errors are immediately collapsed
into a single final bound.
As a result, the framework can distinguish mathematically equivalent
algorithmic variants that organize their computational kernels differently.
The analysis of triangular solves and LU factorization illustrates this point.
It identifies which operations contribute to the final error and how those
contributions depend on the particular algorithmic variant.

A further advantage is the modularity, i.e., the separation between the structural error analysis
and the error bounds associated with individual computational kernels.
Once the structural analysis has been established, alternative kernel-level
bounds can be substituted without repeating the entire derivation.
For example, bounds for compensated summation or more accurate dot products
can be inserted directly to quantify their effect on the overall algorithm.
It should be noted that this modularity is conditional on the local error 
models in Table~\ref{tab:notation}.
If these structural assumptions change, rather than only the bounds assigned to 
them, the propagation analysis and all resulting error bounds must be revised accordingly.
This modularity also makes the framework suitable for comparing different
precision assignments and arithmetic implementations.
The trackable and modular viewpoint is not limited to rounding error analysis 
and can also be applied to computational and, in particular, communication 
complexity.
By tracking the costs of individual computational kernels and their 
contributions to the overall algorithm, one can identify the main sources of 
computational and communication costs and assess the effects of replacing 
specific kernels.
Such analyses are often more straightforward than rounding error analysis, 
where local errors interact and propagate through subsequent computations in 
a more intricate manner.
We therefore expect this framework to be useful for analyzing and
designing numerical algorithms in increasingly diverse arithmetic
environments on rapidly changing hardware.

\section*{Acknowledgments}
The authors thank Bowen Gao and Zhuang-Ao He for helpful discussions.

M.~Shao is supported in part by the National Key R\&D Program of China under grant No.~2023YFB3001603.
E.~Carson and Y.~Ma are supported by the European Union (ERC, inEXASCALE, 101075632).
Views and opinions expressed are those of the authors only and do not necessarily reflect those of the European Union or the European Research Council.
Neither the European Union nor the granting authority can be held responsible for them.
E.~Carson and Y.~Ma additionally acknowledge support from the Charles University Research Centre program No. UNCE/24/SCI/005.

\bibliographystyle{plainurl}
\bibliography{ref}

\appendix
\section{Proofs for Section~\ref{example:triangular-solver}}
\label{appendix:example:triangular}

\subsection{Proofs for the \texttt{XDOT}-based algorithm}
\label{subsec:triangular-xdot}
\begin{proof}[Proof of Lemma~\ref{lem:triangular-solver-xdot}]
For \(i<n\),~\eqref{eq:trsv-dot:div}--\eqref{eq:trsv:sub}
give
\begin{equation}
\begin{split}
&b_i-\bigl(U_{i,i+1:n}+\Delta U^{(\mathtt{dot})}_{i,i+1:n}(n-i)\bigr)\hat x_{i+1:n}
=(1+\delta_i^{(\mathtt{add})})^{-1}(1+\delta_i^{(\mathtt{div})})^{-1}U_{i,i}\hat x_i.
\end{split}
\end{equation}
Then defining \(\Delta U_{i,j}^{(\mathtt{trsv})}(n)\) by~\eqref{eq:lem-trsv-dot-trackable}
proves \((U+\Delta U^{(\mathtt{trsv})}(n))\hat x=b\).
The same argument works for \(i=n\) with
\(\delta_n^{(\mathtt{add})}=0\) and with no dot-product term.
\end{proof}

\subsection{Proofs for the \texttt{XAXPY}-based algorithm}
\label{subsec:triangular-xaxpy}
\begin{proof}[Proof of Lemma~\ref{lem:triangular-solver-xaxpy}]
Repeated application of~\eqref{eq:trsv:axpy} gives
\begin{equation}
\begin{split}
\hat s_i
={}&(1+\delta_{i,i+1}^{(\mathtt{add})})\cdots
(1+\delta_{in}^{(\mathtt{add})})b_i\\
&-\sum_{j=i+1}^n
(1+\delta_{i,i+1}^{(\mathtt{add})})\cdots
(1+\delta_{ij}^{(\mathtt{add})})
(1+\delta_{ij}^{(\mathtt{mult})})U_{i,j}\hat x_j.
\end{split}
\label{eq:lem-proof:hatsi}
\end{equation}
Moreover, from~\eqref{eq:trsv-dot:div},
\begin{equation}
    \widehat{s}_i=(1+\delta_i^{(\mathtt{div})})^{-1}U_{i,i}\widehat{x}_i.
\end{equation}
Substituting this relation into~\eqref{eq:lem-proof:hatsi} and multiplying
\((1+\delta_{in}^{(\mathtt{add})})^{-1}\cdots
(1+\delta_{i,i+1}^{(\mathtt{add})})^{-1}\) on the left yields
\begin{equation}
\begin{split}
b_i={}&(1+\delta_{in}^{(\mathtt{add})})^{-1}\cdots
(1+\delta_{i,i+1}^{(\mathtt{add})})^{-1}
(1+\delta_i^{(\mathtt{div})})^{-1}U_{i,i}\hat x_i\\
&+\sum_{j=i+1}^{n}
(1+\delta_{in}^{(\mathtt{add})})^{-1}\cdots
(1+\delta_{i,j+1}^{(\mathtt{add})})^{-1}
(1+\delta_{ij}^{(\mathtt{mult})})U_{i,j}\hat x_j.
\end{split}
\label{eq:lem-proof:Uii+Uij}
\end{equation}
Then defining \(\Delta U_{i,j}^{(\mathtt{trsv})}(n)\) by~\eqref{eq:lem-trsv-axpy-trackable}
gives \((U+\Delta U^{(\mathtt{trsv})}(n))\hat x=b\).
\end{proof}

\section{Proofs for Section~\ref{example:lu}}

\subsection{Proof for the right-looking block LU algorithm}
\label{appendix:example:lu-proof-right}
\begin{proof}[Proof of Theorem~\ref{thm:block-lu-right}]
Repeated application of~\eqref{eq:lem4:blocklu:update-trail} gives
\begin{equation}
\begin{split}
    \Acalhat^{(i-1)}_{i:q,i}={}& \Acal_{i:q,i}
    - \sum_{k=1}^{i-1}\LLhat_{i:q,k}\UUhat_{k,i}
    - \sum_{k=1}^{i-1}\bigl(\Delta(\LL_{i:q,k}\UU_{k,i})^{(\mathtt{gemm})}(n_b)\bigr)^{(k)}
    - \sum_{k=0}^{i-2}\bigl(\Delta\Acal^{(k)}_{i:q,i}\bigr)^{(\mathtt{gemm})}.
    \label{eq:blocklu:update-expanded}
\end{split}
\end{equation}
Combining this identity~\eqref{eq:blocklu:update-expanded} 
with~\eqref{eq:lem4:blocklu:trf2} gives~\eqref{eq:blocklu:getrf-trackable}.
Repeated application of~\eqref{eq:lem4:blocklu:update-trail} applied to the 
off-diagonal part of the \(i\)th block row gives
\begin{equation}
\begin{split}
\Acalhat^{(i-1)}_{i,i+1:p}
={}&\Acal_{i,i+1:p}
-\sum_{k=1}^{i-1}\LLhat_{i,k}\UUhat_{k,i+1:p}
-\sum_{k=1}^{i-1}
\bigl(\Delta(\LL_{i,k}\UU_{k,i+1:p})^{(\mathtt{gemm})}(n_b)\bigr)^{(k)} \\
&-\sum_{k=0}^{i-2}
\bigl(\Delta\Acal^{(k)}_{i,i+1:p}\bigr)^{(\mathtt{gemm})}.
\label{eq:blocklu:update-expanded-off}
\end{split}
\end{equation} 
On the other hand,~\eqref{eq:lem4:blocklu:trsm} implies
\[
\Acalhat^{(i-1)}_{i,i+1:p}
=\LLhat_{i,i}\UUhat_{i,i+1:p}
-\Delta(\LL_{i,i}\UU_{i,i+1:p})^{(\mathtt{trsv})}(n_b).
\] 
Together with~\eqref{eq:blocklu:update-expanded-off} and~\eqref{eq:blocklu:getrf-trackable}, we have~\eqref{eq:blocklu:getrf-trsv} and~\eqref{eq:blocklu:getrf}.

When \(n_b=1\), from~\eqref{eq:blocklu:div}, we have
\begin{equation}
\Acalhat^{(i-1)}_{j,i} + \bigl(1-(1+\delta_{j,i}^{(\mathtt{div})})^{-1}\bigr)
\LLhat_{j,i}\UUhat_{i,i} = \LLhat_{j,i}\UUhat_{i,i},
\end{equation}
which concludes the proof for \(n_b=1\) by using~\eqref{eq:blocklu:getrf-trackable},~\eqref{eq:blocklu:getrf-trsv}, and~\eqref{eq:blocklu:getrf}.
\end{proof}

\subsection{Proof for the left-looking block LU algorithm}
\label{appendix:example:lu-proof-left}
\begin{proof}[Proof of Theorem~\ref{thm:block-lu-left}]
    From~\eqref{eq:lem4:blocklu-2:gemm}, the current block column satisfies
    \begin{align*}
        \Acalhat_{i:q,i}
        ={}& \Acal_{i:q,i}
        -\LLhat_{i:q,1:i-1}\UUhat_{1:i-1,i}
        -\Delta\bigl(\LL_{i:q,1:i-1}\UU_{1:i-1,i}\bigr)^{(\mathtt{gemm})}
        \bigl((i-1)n_b\bigr) \\
        &-\bigl(\Delta\Acal_{i:q,i}\bigr)^{(\mathtt{gemm})}.
    \end{align*}
    Combining this relation with~\eqref{eq:lem4:blocklu:trf2} gives
    \eqref{eq:blocklu-2:getrf-trackable}.
    For the off-diagonal part of the \(i\)th block row,~\eqref{eq:lem4:blocklu-2:row-gemm} and~\eqref{eq:lem4:blocklu:trsm} give
    \begin{align*}
        \Acal_{i,i+1:p}={}&\LLhat_{i,1:i-1}\UUhat_{1:i-1,i+1:p}
        +\LLhat_{i,i}\UUhat_{i,i+1:p} \\
        &+\Delta(\LL_{i,1:i-1}\UU_{1:i-1,i+1:p})^{(\mathtt{gemm})}
        \bigl((i-1)n_b\bigr)\\
        &+\bigl(\Delta\Acal_{i,i+1:p}\bigr)^{(\mathtt{gemm})}
        -\Delta(\LL_{i,i}\UU_{i,i+1:p})^{(\mathtt{trsv})}(n_b),
    \end{align*}
    which proves~\eqref{eq:blocklu-2:getrf-trsv} by rearranging this equality.

    Finally, for \(n_b=1\), similarly to the proof for 
    Algorithm~\ref{alg:block-LU} with \(n_b=1\), by~\eqref{eq:blocklu:div}, 
    we can conclude the proof.
\end{proof}

\subsection{Proof of Corollary~\ref{cor:block-lu-componentwise}}
\label{appendix:cor:block-lu-componentwise}
\begin{proof}[Proof of Corollary~\ref{cor:block-lu-componentwise}]
We apply the assumptions in Tables~\ref{tab:notation}
and~\ref{tab:notation-original} directly to the perturbation expressions in
Theorems~\ref{thm:block-lu-right} and~\ref{thm:block-lu-left}.
It suffices to consider the off-diagonal block row \(i,i+1:p\).
The block column case follows in the same way, with the panel
factorization error in place of the triangular solve error.

We first consider the right-looking algorithm, i.e., Algorithm~\ref{alg:block-LU}.
For \(1\leq i\leq p-1\) and \(1\leq k\leq i-1\),
Table~\ref{tab:notation-original} gives
\begin{equation}
\begin{split}
\abs{\bigl(\Delta(\LL_{i,k}\UU_{k,i+1:p})^{(\mathtt{gemm})}(n_b)\bigr)^{(k)}}
&\leq \epsgemm(n_b)\abs{\LLhat_{i,k}}\abs{\UUhat_{k,i+1:p}},\\
\abs{\bigl(\Delta\Acal_{i,i+1:p}^{(k-1)}\bigr)^{(\mathtt{gemm})}}
&\leq \epsgemmadd\abs{\Acalhat_{i,i+1:p}^{(k-1)}},\\
\abs{\Delta(\LL_{i,i}\UU_{i,i+1:p})^{(\mathtt{trsv})}(n_b)}
&\leq \epstrsv(n_b)\abs{\LLhat_{i,i}}\abs{\UUhat_{i,i+1:p}}.
\end{split}
\label{eq:cor-proof-general-gemm-bounds}
\end{equation}
The matrix-addition perturbations above depend on the intermediate
blocks \(\Acalhat_{i,i+1:p}^{(k-1)}\).
To bound their contributions in Theorem~\ref{thm:block-lu-right}
by \(\abs{\LLhat}\abs{\UUhat}\), we first establish the relationship between
\(\Acalhat_{i,i+1:p}^{(k-1)}\) and \(\abs{\LLhat_{i,i}}\abs{\UUhat_{i,i+1:p}}\).
Combining~\eqref{eq:lem4:blocklu:update-trail} 
and~\eqref{eq:lem4:blocklu:trsm} with~\eqref{eq:cor-proof-general-gemm-bounds} gives
\begin{equation}
\begin{split}
(1-\epsgemmadd)\abs{\Acalhat_{i,i+1:p}^{(k-1)}}
&\leq \abs{\Acalhat_{i,i+1:p}^{(k)}}
 +(1+\epsgemm(n_b))\abs{\LLhat_{i,k}}\abs{\UUhat_{k,i+1:p}},\\
\abs{\Acalhat_{i,i+1:p}^{(i-1)}}
&\leq (1+\epstrsv(n_b))
\abs{\LLhat_{i,i}}\abs{\UUhat_{i,i+1:p}}.
\end{split}
\label{eq:cor-proof-local}
\end{equation}
Since \(\epsgemmadd<1\), repeated substitution, starting with
the second inequality, gives
\begin{equation}
\begin{split}
\abs{\Acalhat_{i,i+1:p}^{(k-1)}}
\leq{}&
\frac{1+\max\{\epsgemm(n_b),\epstrsv(n_b)\}}
{(1-\epsgemmadd)^{i-k}}
(\abs{\LLhat}\abs{\UUhat})_{i,i+1:p},
\qquad 1\leq k\leq i-1.
\end{split}
\label{eq:cor-proof-input}
\end{equation}
Substituting these bounds directly into~\eqref{eq:blocklu:getrf-trsv} gives
\begin{equation}
\begin{split}
\abs{\Delta(\LL\UU)_{i,i+1:p}^{(\mathtt{getrf})}(n,n_b)}
\leq{}&\epsgemm(n_b)
\sum_{k=1}^{i-1}\abs{\LLhat_{i,k}}\abs{\UUhat_{k,i+1:p}}\\
&+\epsgemmadd\sum_{k=0}^{i-2}
\abs{\Acalhat_{i,i+1:p}^{(k)}}
+\epstrsv(n_b)\abs{\LLhat_{i,i}}\abs{\UUhat_{i,i+1:p}}\\
\leq{}&\left(
\frac{1+\max\{\epsgemm(n_b),\epstrsv(n_b)\}}
{(1-\epsgemmadd)^{i-1}}-1
\right)(\abs{\LLhat}\abs{\UUhat})_{i,i+1:p}.
\end{split}
\label{eq:cor-proof-iteration}
\end{equation}
Applying the same argument to~\eqref{eq:blocklu:getrf-trackable} replaces
\(\epstrsv(n_b)\) by \(\epstrfvec(n_b)\).
There are at most \(p-1\) preceding updates, so these two bounds
imply~\eqref{eq:blocklu:getrf-componentwise} and~\eqref{eq:blocklu:eta}.
For \(p=1\), only the panel factorization is needed.

We then consider the left-looking algorithm, i.e., Algorithm~\ref{alg:block-LU-2}.
Using the monotonicity of \(\epsgemm\), Table~\ref{tab:notation-original}
and~\eqref{eq:blocklu-2:getrf-trsv} directly give
\begin{equation}
\begin{split}
\abs{\Delta(\LL\UU)_{i,i+1:p}^{(\mathtt{getrf})}(n,n_b)}
\leq{}&\epsgemm((p-1)n_b)
\abs{\LLhat_{i,1:i-1}}\abs{\UUhat_{1:i-1,i+1:p}}\\
&+\epsgemmadd\abs{\Acal_{i,i+1:p}}
+\epstrsv(n_b)\abs{\LLhat_{i,i}}\abs{\UUhat_{i,i+1:p}}.
\end{split}
\end{equation}
Since \eqref{eq:blocklu:getrf} implies
\[
\abs{\Acal_{i,i+1:p}}\leq
(\abs{\LLhat}\abs{\UUhat})_{i,i+1:p}
+\abs{\Delta(\LL\UU)_{i,i+1:p}^{(\mathtt{getrf})}(n,n_b)},
\]
we obtain
\begin{equation}
\begin{split}
&(1-\epsgemmadd)
\abs{\Delta(\LL\UU)_{i,i+1:p}^{(\mathtt{getrf})}(n,n_b)}\\
&\quad\leq
\bigl(\max\{\epsgemm((p-1)n_b),\epstrsv(n_b)\}
+\epsgemmadd\bigr)
(\abs{\LLhat}\abs{\UUhat})_{i,i+1:p}.
\end{split}
\label{eq:cor-proof-left}
\end{equation}
Dividing by \(1-\epsgemmadd\) gives the off-diagonal block-row
bound in~\eqref{eq:blocklu-2:getrf-componentwise-monotone}.
Applying the same argument to
\eqref{eq:blocklu-2:getrf-trackable} yields the estimate with
\(\epstrfvec(n_b)\) in place of \(\epstrsv(n_b)\).
For \(p=1\), only the panel bound is required.

For \(n_b=1\), the scalar triangular solve is exact and the
subdiagonal multipliers are computed by division.
From Table~\ref{tab:notation-original}, we have
\begin{equation}
\abs{\Delta(\LL_{ij}\UU_{jj})^{(\mathtt{div})}}
\leq\frac{\epsdiv}{1-\epsdiv}
\abs{\LLhat_{ij}}\abs{\UUhat_{jj}},\qquad i>j.
\label{eq:cor-proof-scalar-div-bound}
\end{equation}
We then set \(n_b=1\), \(p=n\), replace the terminal triangular solve error 
bound by zero, and use \(\epsdiv/(1-\epsdiv)\) for the panel bound below the 
diagonal.
The right-looking and left-looking estimates then give, 
respectively,~\eqref{eq:blocklu:getf2-componentwise} and~\eqref{eq:blocklu-2:getf2-componentwise}.
\end{proof}

\section{Proofs for Section~\ref{subsec:diff-setting}}
Throughout this subsection, \(\gamma_n=(1-\macheps)^{-n}-1\), so that
\(\gamma_j+\gamma_k+\gamma_j\gamma_k=\gamma_{j+k}\).
\subsection{Proofs for complex arithmetic (Section~\ref{remark:ourgamman})}
\label{appendix:example:lu-ourgamman}
\begin{proof}[Proof of Lemma~\ref{lem:complex}]
Similarly to the proof of~\cite[Lemma~3.5]{Higham2002}, we can 
prove~\eqref{eq:lem-complex-er:add-minus} and~\eqref{eq:lem-complex-er:times}.
It remains to prove~\eqref{eq:lem-complex-er:div}.
The proof for~\eqref{eq:lem-complex-er:div} also follows that 
of~\cite[Lemma~3.5]{Higham2002}, except that we additionally account for one rounding error arising from the final division.
\end{proof}

\subsection{Proofs for quaternion arithmetic (Section~\ref{remark:quaternion})}
\label{appendix:example:lu-quaternion}

\begin{proof}[Proof of Lemma~\ref{lem:quaternion}]
It is straightforward to obtain~\eqref{eq:lem-quaternion-er:add-minus} with the
relative perturbation on either side.
For multiplication, each component is a real dot product of length \(4\), which 
gives the stated bound for a right relative perturbation.
Since every nonzero quaternion \(z\) is invertible, from
\(\fl(z)=z(1+\delta_r)\) we may equivalently write \(\fl(z)=(1+\delta_\ell)z\), 
where \(\delta_\ell=z\delta_r z^{-1}\).
The quaternion norm is multiplicative, so \(\abs{\delta_\ell}=\abs{\delta_r}\).  
Thus the same bound holds on either side.
This proves~\eqref{eq:lem-quaternion-er:times}; the zero case is immediate.

It remains to prove~\eqref{eq:lem-quaternion-er:div}.
Assume that \(\delta_i\) is a number satisfying \(\abs{\delta_i}\leq \macheps\)
throughout this proof.
From~\eqref{eq:quaternion:div}, we have
\begin{equation}
\begin{split}
\fl\bigl(y_1^2+y_2^2+y_3^2+y_4^2\bigr)&=(y_1^2+y_2^2+y_3^2+y_4^2)(1+\theta_4)
\end{split}
\end{equation}
with \((1-\macheps)^4\leq \abs{1+\theta_4}\leq (1+\macheps)^4\).
Then
\begin{equation}
\begin{split}
\fl\Bigl(\frac{x\Bar{y}}{y_1^2+y_2^2+y_3^2+y_4^2}\Bigr) &= \frac{x\Bar{y}(1+\tilde\theta_4)}
{(y_1^2+y_2^2+y_3^2+y_4^2)(1+\theta_4)}(1+\delta_1) \\
&= \frac{x\Bar{y}}{y_1^2+y_2^2+y_3^2+y_4^2}(1+\theta_9)
\end{split}
\end{equation}
with
\begin{equation}
\begin{split}
\abs{\theta_9} &= \abs{\frac{1+\tilde\theta_4}{1+\theta_4}(1+\delta_1)-1}
\leq (1+2\gamma_4)(1+\gamma_5)-1
\leq \gamma_{13}.
\end{split}
\end{equation}
Analogous to the proof of \cite[Lemma~3.5]{Higham2002}, we first obtain the
right-relative-error representation.
Applying the same similarity argument as for multiplication, with \(z=xy^{-1}\), 
gives the left-relative-error representation with exactly the same norm bound.  
This completes the proof.
\end{proof}

\subsection{Proofs for Kahan compensated summation (Section~\ref{remark:Kahan})}
\label{appendix:example:lu-Kahan}
\begin{lemma} \label{lem:trsv-Kahan}
Let \(U\in\mathbb R^{n\times n}\) be a triangular matrix and
\(b\in\mathbb R^{n}\) be a right-hand side vector.
Assume that the triangular system is solved by 
Algorithm~\ref{alg:trsv-dot}, where the dot product in
Line~\ref{line:trsv-dot:dot} is based on Kahan compensated summation.
Then there exists a triangular matrix \(\Delta U^{(\mathtt{trsv})}(n)\),
with the same triangular structure as \(U\), such that the computed result
\(\hat{x}\) satisfies
\begin{equation} \label{eq:lem:trsv-Kahan}
    \bigl(U+\Delta U^{(\mathtt{trsv})}(n)\bigr)\hat x=b,
    \qquad
    \abs{\Delta U^{(\mathtt{trsv})}(n)}
    \leq \epstrsv(n)\abs U,
\end{equation}
where \(\epstrsv(n) = 3\macheps+\bigO(n\macheps^2)\).
\end{lemma}

\begin{proof}
From Lemma~\ref{lem:triangular-solver-xdot}, we have \((U+\Delta U^{(\mathtt{trsv})}(n))\hat{x}=b\).
For real arithmetic, the componentwise assumptions in
Table~\ref{tab:notation-original} give
\[
 \abs{\Delta U^{(\mathtt{trsv})}(n)}
 \leq \max\{\epsdot(n-1),\gamma_2\}\abs U,
\]
where the term \(\gamma_2\) bounds the diagonal perturbation generated by
one subtraction and one division.  Hence it remains only to estimate
\(\epsdot\) for the Kahan-compensated dot product.

From the analysis of Kahan compensated summation given by~\cite[(4.9)]{Higham2002}, we have
\[
\abs{\fl\bigl( \sum_{i=1}^n\fl(x_iy_i) \bigr) -  \sum_{i=1}^n\fl(x_iy_i)}
\leq \bigl(2\macheps+\bigO(n\macheps^2)\bigr)\sum_{i=1}^n\abs{\fl(x_iy_i)},
\]
where \(\fl(x_iy_i)\) satisfies that
\begin{equation*}
    \fl(x_iy_i) = x_iy_i(1+\delta_i)
    = (x_i+\Delta x_i)y_i, \qquad \abs{\delta_i}\leq \macheps\quad\text{and}\quad
    \abs{\Delta x_i}\leq \macheps\abs{x_i}.
\end{equation*}
Noticing that \(\fl\bigl( \sum_{i=1}^n\fl(x_iy_i) \bigr) = \fl(x\trans y)\), 
it then follows that
\begin{equation*}
\begin{split}
    \abs{\fl(x\trans y) - (x+\Delta x)\trans y}
    &\leq \bigl(2\macheps+\bigO(n\macheps^2)\bigr)(1+\macheps) \sum_{i=1}^n\abs{x_iy_i} \\
    &\leq \bigl(2\macheps+\bigO(n\macheps^2)\bigr) \sum_{i=1}^n\abs{x_iy_i} .
\end{split}
\end{equation*}
Furthermore, moving the term \(\Delta x\trans y\) on the left-hand side to 
the right-hand side, we have
\begin{equation}
\begin{split}
    \abs{\fl(x\trans y) - x\trans y}
    &\leq \bigl(2\macheps+\bigO(n\macheps^2)\bigr) \sum_{i=1}^n\abs{x_iy_i}
    + \macheps \sum_{i=1}^n\abs{x_i}\abs{y_i} \\
    &\leq \bigl(3\macheps+\bigO(n\macheps^2)\bigr) \abs{x}\trans\abs{y}.
\end{split}
\end{equation}
Therefore \(\epstrsv(n)=3\macheps+\bigO(n\macheps^2)\).
\end{proof}


\end{document}